\documentclass[12pt]{elsarticle}

\usepackage[utf8]{inputenc}
\usepackage[T1]{fontenc}
\usepackage{amsmath}
\numberwithin{equation}{section}
\usepackage{amsfonts}
\usepackage{amssymb}
\usepackage[version=4]{mhchem}
\usepackage{stmaryrd}
\usepackage{graphicx}
\usepackage{verbatim}
\usepackage[export]{adjustbox}
\usepackage{float}
\usepackage{subcaption}
\usepackage[hidelinks]{hyperref}
\usepackage{url}
\usepackage{placeins}
\usepackage{todonotes}
\usepackage{amsthm}
\newtheorem{theorem}{Theorem}[section]
\usepackage{algorithm}
\usepackage{algpseudocode} 
\usepackage{amsthm}
\usepackage{booktabs}
\newtheorem{remark}{Remark}
\newtheorem{lemma}{Lemma}

\usepackage{graphicx}
\graphicspath{{figures/}}
\usepackage{tabularx}
\usepackage{multirow}

\usepackage{microtype}

\begin{document}

\begin{frontmatter}



\title{An Efficient Constant-Coefficient MSAV Scheme for Computing Vesicle Growth and Shrinkage}

\author[iit]{Zhiwei Zhang}
\ead{zzhang107@hawk.illinoistech.edu}

\author[iit]{Shuwang Li}
\ead{sli15@illinoistech.edu}

\author[uci]{John Lowengrub}
\ead{jlowengr@uci.edu}

\author[utk]{Steven M. Wise}
\ead{swise1@utk.edu}

\affiliation[iit]{organization={Department of Applied Mathematics, Illinois Institute of Technology},
            city={Chicago},
            postcode={60616},
            state={IL},
            country={USA}}

\affiliation[uci]{organization={Department of Mathematics, University of California, Irvine},
            city={Irvine},
            postcode={92697},
            state={CA},
            country={USA}}

\affiliation[utk]{organization={Department of Mathematics, The University of Tennessee},
            city={Knoxville},
            postcode={37996},
            state={TN},
            country={USA}}

\begin{abstract}
We present a fast, unconditionally energy-stable numerical scheme for simulating vesicle deformation under osmotic pressure using a phase-field approach. The model couples an Allen--Cahn equation for the biomembrane interface with a variable-mobility Cahn--Hilliard equation governing mass exchange across the membrane. Classical approaches, including nonlinear multigrid and Multiple Scalar Auxiliary Variable (MSAV) methods, require iterative solution of variable-coefficient systems at each time step, resulting in substantial computational cost. We introduce a constant-coefficient MSAV (CC-MSAV) scheme that incorporates stabilization directly into the Cahn--Hilliard evolution equation rather than the chemical potential. This reformulation yields fully decoupled constant-coefficient elliptic problems solvable via fast discrete cosine transform (DCT), eliminating iterative solvers entirely. The method achieves $O(N^2 \log N)$ complexity per time step while preserving unconditional energy stability and discrete mass conservation. Numerical experiments verify second-order temporal and spatial accuracy, mass conservation to relative errors below $5 \times 10^{-11}$, and close agreement with nonlinear multigrid benchmarks. On grids with $N \geq 2048$, CC-MSAV achieves 6--15$\times$ overall speedup compared to classical MSAV with optimized preconditioning,  while the dominant Cahn--Hilliard subsystem is accelerated by up to two orders of magnitude. These efficiency gains, achieved without sacrificing accuracy, make CC-MSAV particularly well-suited for large-scale simulations of vesicle dynamics.
\end{abstract}

\begin{keyword}
Cahn-Hilliard equation
\sep Scalar auxiliary variable
\sep Constant-coefficient
\sep Energy stability
\sep Variable mobility
\sep Phase-field method
\sep Vesicle deformation
\end{keyword}

\end{frontmatter}


\section{Introduction}

Lipid vesicles are closed bilayer membranes enclosing an aqueous volume. These structures form when phospholipid molecules self-assemble in aqueous environments, with hydrophilic heads facing outward and hydrophobic tails forming the interior bilayer. Due to their structural simplicity and rich mechanical behavior, vesicles provide tractable platforms for studying membrane deformation, shape transitions, and transport phenomena relevant to cellular biology \cite{seifert1997configurations}.  
Among the many physical 
processes governing vesicle behavior, osmotically driven volume change plays a particularly 
important role: concentration gradients across the semi-permeable membrane induce material 
transport and subsequent morphological evolution \cite{mori2011model,vogl2014effect,quaife2021hydrodynamics}. Understanding how vesicles respond to osmotic stress, whether growing, shrinking, or maintaining equilibrium, is essential for applications ranging from drug delivery systems \cite{elani2014vesicle} to synthetic cell design and fundamental biophysics.

Mathematical modeling of vesicle mechanics has evolved along two complementary approaches. Sharp-interface models explicitly track the membrane as a moving interface, satisfying geometric constraints and interface conditions exactly. These methods have successfully captured vesicle dynamics in flow \cite{veerapaneni2009boundary,veerapaneni2009numerical,sohn2010dynamics,salac2011level,hausser2013thermodynamically}, shape transitions under shear \cite{liu2017dynamics}, and membrane wrinkling phenomena \cite{liu2014nonlinear,xiao2023three}. Phase-field models, alternatively, represent the membrane as a continuous transition layer of small thickness $\varepsilon$, transforming the moving interface problem into coupled reaction–diffusion equations on a fixed domain \cite{du2004phase,wang2008modelling,lowengrub2009phase}. The phase-field approach implicitly determines the interface position and offers advantages in computational simplicity, as well as natural handling of topological changes.

Our work builds upon the phase-field vesicle model of Tang et al.~\cite{tang2023phase}, which couples Allen--Cahn and Cahn--Hilliard equations to capture membrane evolution and concentration dynamics under osmotic pressure. While their nonlinear multigrid (NLMG) solver achieves first-order temporal accuracy with near-optimal convergence, it requires iterative solution of the coupled nonlinear system at each time step, incurring substantial computational expense for large-scale simulations. 

Several recent developments have enriched the numerical analysis of related phase-field problems. Guo et al.~\cite{guo2025phase} constructed high-order energy-stable discontinuous Galerkin methods for osmotic flow through semi-permeable membranes, and Zhang et al.~\cite{zhang2024existence} established existence and regularity results for triplet phase-separation models. For gradient flows, the Scalar Auxiliary Variable (SAV) framework \cite{shen2018scalar}, building upon the earlier Invariant Energy Quadratization (IEQ) approach \cite{yang2017numerical}, has emerged as an effective tool for constructing energy-stable schemes. SAV introduces auxiliary variables to represent nonlinear energy terms, enabling reformulation of the system into linear problems with constant coefficients while maintaining unconditional energy stability. The Multiple-SAV (MSAV) extension \cite{cheng2018multiple} introduces separate auxiliary variables for each nonlinear contribution and serves as the foundation for the present work. However, extending these advantages to Cahn--Hilliard equations with variable mobility 
is nontrivial and leads to additional structural challenges.

For Cahn--Hilliard equations with variable mobility, applying SAV- and MSAV-type 
reformulations remains challenging, as unconditional energy stability is often 
difficult to maintain together with a constant-coefficient, fully decoupled linear 
structure. In classical SAV and MSAV approaches, stabilization is typically introduced 
at the level of the chemical potential; when the mobility is variable, this generally 
leads to linear equations with variable coefficients and thus 
precludes the constant-coefficient structure of the resulting elliptic problems. Huang et al.~\cite{huang2023structure} developed a structure-preserving upwind-SAV 
scheme for degenerate mobility that is energy stable; however, the mobility enters 
the discrete fluxes at the new time level, resulting in mobility-dependent operators 
rather than constant-coefficient elliptic problems. Bretin et al.~\cite{bretin2023mobility} 
showed that straightforward SAV extensions lose Fourier-explicit solvability for 
degenerate mobilities. To address this, they introduced a mobility-based SAV 
reformulation in which the implicit operator involves a constant mobility upper bound, 
yielding constant-coefficient linear subproblems and stability results that rely on 
auxiliary-variable constraints. By contrast, Orizaga and Witelski~\cite{orizaga2024imex} 
proposed IMEX splittings that enable constant-coefficient implicit solves by separating 
the mobility into constant and variable parts, but unconditional energy stability is 
not guaranteed in general and holds only for certain parameter choices.

Motivated by these challenges, we introduce a constant-coefficient Multiple Scalar 
Auxiliary Variable (CC-MSAV) scheme that exploits the specific structure of the coupled 
Allen--Cahn--Cahn--Hilliard system for vesicle deformation. The model couples an 
Allen--Cahn equation for the membrane phase field $\phi$ with a Cahn--Hilliard equation 
for the concentration field $\psi$, where variable mobility $M_\psi(\phi)$ introduces 
coupling between the two equations. Following classical approaches, we treat the mobility 
through extrapolation to achieve linearity. Our key insight is to reformulate the osmotic 
coupling such that stabilization can be incorporated directly into the Cahn--Hilliard 
evolution equation rather than into the chemical potential, which is the standard approach in 
classical MSAV methods. Specifically, the energy functional decomposes into four components 
(surface, bending, surface area constraint, and osmotic energy), each with a nonlinear part represented by a 
separate SAV variable. The osmotic energy, which couples $\phi$ and $\psi$, is wrapped 
entirely within its SAV variable, ensuring that the chemical potential contains no linear differential operators acting on the unknown $\psi^{n+1}$, such as $\psi^{n+1}$ or $\Delta\psi^{n+1}$. This formulation enables linear 
stabilization to be incorporated directly into the evolution equation, yielding a fully 
decoupled system at each time step: five fourth-order equations for the Allen--Cahn subsystem 
and two second-order equations for the Cahn--Hilliard subsystem, all with constant 
coefficients. Each equation is solved via discrete cosine transform (DCT), which diagonalizes 
the discrete Laplacian under homogeneous Neumann boundary conditions and provides fast, 
machine-precision direct solves with computational complexity $O(N^2 \log N)$ 
\cite{strang1999discrete,trefethen2000spectral,leveque2007finite,schumann1988fast}.

This constant-coefficient reformulation provides substantial computational advantages over
variable-coefficient approaches. At grid resolutions relevant for vesicle simulations
($N \ge 2048$), sparse direct solvers for variable-coefficient elliptic equations become
prohibitively expensive due to fill-in during factorization, resulting in superlinear
computational complexity and excessive memory usage. Iterative methods such as multigrid
or preconditioned Krylov solvers offer improved scalability, but even well-optimized
implementations remain significantly slower than fast transform-based solvers available
for constant-coefficient operators. When classical MSAV is equipped with a DCT-based Poisson preconditioner, approximately
12--13 Krylov iterations are required per Cahn--Hilliard solve (24--26 per time step).
In contrast, the proposed CC-MSAV formulation yields fully decoupled constant-coefficient
elliptic problems that are solved directly via DCT with zero iterations. As demonstrated
in Section~\ref{sec:computational_complexity}, this structural simplification leads to
6--15$\times$ overall speedup for $N \ge 2048$. Moreover, the Cahn--Hilliard subsystem, which dominates the computational cost in classical formulations, becomes a minor component of the total runtime under the proposed scheme. This reduction yields order-of-magnitude savings relative to both sparse direct solvers and well-preconditioned iterative methods, without compromising accuracy or unconditional energy stability.

The numerical experiments demonstrate that the proposed CC-MSAV method achieves second-order 
temporal and spatial accuracy, robust energy decay, and near machine precision mass 
conservation, with relative errors below $5 \times 10^{-11}$. The method performs robustly across a wide range of scenarios, 
including non-convex geometries (triangles, stars, crescents)
involving topological changes. Comparisons with classical MSAV and well-preconditioned 
iterative methods demonstrate significant 
reductions in wall-clock time while maintaining agreement with established nonlinear 
multigrid benchmarks. The combination of unconditional energy stability, second-order 
accuracy, optimal complexity scaling, and the elimination of iterative solvers makes CC-MSAV 
particularly well-suited for large-scale computational investigations of vesicle dynamics 
under osmotic stress.

The remainder of this paper is organized as follows. In Section~\ref{sec:mathematical_formulation_phase_field_model}, we present the phase-field model and derive the MSAV reformulation. In Section~\ref{sec:numerical_scheme_development}, we present the numerical discretization, prove energy stability, and discuss computational efficiency. In Section~\ref{sec:numerical_results}, we provide comprehensive numerical validation, including convergence studies, benchmark comparisons with NLMG, and robustness demonstrations on complex geometries. In Section~\ref{sec:conclusion_and_future_work}, we summarize the findings and discuss future research directions. For completeness, in~\ref{sec:appendix_energy_stability_proof} we provide the full proof of unconditional energy stability for the proposed CC-MSAV scheme; in~\ref{sec:appendix_solution_procedure} we give the complete solution procedure for the MSAV--BDF2 system; and in~\ref{sec:appendix_fast_direct_solver} we describe the DCT-based fast direct solver used for the constant-coefficient elliptic equations.

\section{Mathematical Formulation of the Phase-Field Model}
\label{sec:mathematical_formulation_phase_field_model}

In this section, we present the mathematical formulation of the phase-field vesicle model and its MSAV reformulation. We first recall the model of Tang et al. \cite{tang2023phase}, then introduce an MSAV reformulation that preserves energy dissipation while leading to linear, decoupled, and computationally efficient schemes.

\subsection{Phase-Field Vesicle Model}
\label{sec:benchmark_model}

We consider the phase-field vesicle model of Tang~et~al.~\cite{tang2023phase}, which describes vesicle deformation driven by surface tension, bending elasticity, osmotic pressure, and an arc-length conservation constraint. 
Let $\Omega \subset \mathbb{R}^2$ denote a square computational domain. 
Two primary fields are defined on $\Omega$: the phase-field variable $\phi(x,t): \Omega \to \mathbb{R}$, with $\phi = 1$ inside the vesicle and $\phi = -1$ outside, and the solute concentration $\psi(x,t): \Omega \to \mathbb{R}$. 
The zero level set $\{\mathbf{x}: \phi(\mathbf{x})=0\}$ approximates the membrane location.

\paragraph{Governing equations}
The system evolves according to coupled Allen--Cahn and Cahn--Hilliard equations, representing gradient flows of a total free energy functional $F[\phi,\psi]$:
\begin{align}
    \frac{\partial \phi}{\partial t} &= -M_{\phi}\,\mu, \label{eq:AC}\\
    \frac{\partial \psi}{\partial t} &= \nabla \cdot \!\left( M_{\psi}(\phi)\,\nabla \nu \right), \label{eq:CH}
\end{align}
where $\mu = \delta F / \delta \phi$ and $\nu = \delta F / \delta \psi$ denote the chemical potentials, whose explicit forms will be derived below. Here $M_{\phi} > 0$ is constant, while the Cahn--Hilliard mobility $M_{\psi}(\phi) = 1 - M_{0}(\phi^2 - 1)^2$ with $0 < M_0 < 1$ is large in the bulk phases ($\phi = \pm 1$) but suppressed near the membrane, modeling a semi-permeable interface.

\begin{remark}
Although the evolution equation \eqref{eq:CH} is written in Cahn--Hilliard form, the osmotic free energy in the present model contains no gradient term in $\psi$. As a result, the associated chemical potential $\nu=\delta F/\delta\psi$ contains no spatial derivatives, and the resulting Cahn--Hilliard equation is second order rather than fourth order. We retain the term ``Cahn--Hilliard equation'' to emphasize its conservative gradient-flow structure, consistent with related phase-field vesicle models.
\end{remark}

\paragraph{Free energy functional}
The governing equations~\eqref{eq:AC}--\eqref{eq:CH} are gradient flows of the total free energy
\begin{equation}
    F[\phi,\psi]
    = F^{\text{surf}}[\phi]
    + F^{\text{bend}}[\phi]
    + F^{\text{area}}[\phi]
    + F^{\text{osm}}[\phi,\psi],
    \label{eq:energy_functional}
\end{equation}
which combine surface, bending, arc-length, and osmotic contributions. The energy components are
\begin{align}
    F^{\text{surf}}[\phi] &= \gamma_{\text{surf}} 
        \int_{\Omega} f^{\text{surf}}(\phi, \nabla \phi) \, d\mathbf{x},\\
    F^{\text{bend}}[\phi] &= \gamma_{\text{bend}} 
        \int_{\Omega} f^{\text{bend}}(\phi, \Delta \phi) \, d\mathbf{x},\\
    F^{\text{osm}}[\phi,\psi] &= 
        \int_{\Omega} f^{\text{osm}}(\phi, \psi) \, d\mathbf{x},\\
    F^{\text{area}}[\phi] &= \frac{\gamma_{\text{area}}}{2}
      \left(
        \int_{\Omega} f^{\text{surf}}(\phi, \nabla \phi) \, d\mathbf{x}
        - A
      \right)^{\!2},
\end{align}
where $\gamma_{\text{surf}}$, $\gamma_{\text{bend}}$, and $\gamma_{\text{area}}$ are surface tension, bending rigidity, and arc-length penalty coefficients, respectively.

\noindent
The energy densities are~\cite{cahn1958free,du2005phase,giga2017variational}:
\begin{subequations}\label{eq:energy_densities}
\begin{align}
    f^{\text{surf}}(\phi, \nabla \phi)
        &= \frac{3\sqrt{2}}{4}
           \left( \frac{1}{\varepsilon} g(\phi)
                + \frac{\varepsilon}{2} |\nabla \phi|^2 \right),\\
    f^{\text{bend}}(\phi, \Delta \phi)
        &= \frac{3\sqrt{2}}{16\varepsilon}
           \left( \frac{g'(\phi)}{\varepsilon} - \varepsilon \Delta \phi \right)^2,\\
    f^{\text{osm}}(\phi, \psi)
        &= \frac{1 + p(\phi)}{2} f^{\text{in}}(\psi)
           + \frac{1 - p(\phi)}{2} f^{\text{out}}(\psi),\label{eq:osm_density}
\end{align}
\end{subequations}
where $g(\phi) = \tfrac{1}{4}(\phi^2 - 1)^2$ is the double-well potential, $\varepsilon$ is the interface thickness, and
\begin{equation}
    f^{\text{in/out}}(\psi) = \frac{\gamma_{\text{in/out}}}{2}(\psi - \psi_{\text{in/out}})^2 + \beta_{\text{in/out}}.
\end{equation}
The osmotic energies are smoothly connected by an interpolation function $p(\phi)$ satisfying $p(1)=1$, $p(-1)=-1$, and $p'(1)=p'(-1)=0$. 
In all computations we take $\displaystyle p(\phi) = \sin\!\left(\frac{\pi}{2}\phi\right)$.

\paragraph{Chemical potentials and boundary conditions}
With the energy functional defined, the chemical potentials are obtained as variational derivatives of $F[\phi,\psi]$. 
To circumvent the presence of explicit fourth-order derivatives in $\phi$, an auxiliary variable $\omega:=\frac{1}{\varepsilon} g'(\phi) - \varepsilon \Delta \phi$ is introduced.
The resulting expressions for the chemical potentials are
\begin{align}
    \mu &= \gamma_{\text{surf}} \frac{3 \sqrt{2}}{4}\,\omega
        + \gamma_{\text{bend}} \frac{3 \sqrt{2}}{8}
        \left( \frac{g''(\phi)}{\varepsilon^2} \omega - \Delta \omega \right) \notag\\
        &\quad + \gamma_{\text{area}} 
        \left( \int_{\Omega} f^{\text{surf}}\!\left(\phi,\nabla\phi\right) \, d\mathbf{x} - A \right)
        \frac{3 \sqrt{2}}{4}\, \omega
        + \frac{p'(\phi)}{2}\left( f^{\text{in}}(\psi) - f^{\text{out}}(\psi) \right), \label{eq:mu}\\[2mm]
    \nu &= \frac{1 + p(\phi)}{2} \frac{d f^{\text{in}}}{d \psi}(\psi)
        + \frac{1 - p(\phi)}{2} \frac{d f^{\text{out}}}{d \psi}(\psi). \label{eq:nu}
\end{align}
We impose homogeneous Neumann conditions $\partial_n \phi = \partial_n \mu = \partial_n \omega = \partial_n \psi = \partial_n \nu = 0$ on $\partial\Omega$.

\paragraph{Energy dissipation}
The gradient-flow structure yields the energy dissipation law
\begin{equation}
    \frac{dF}{dt}
    = - \int_{\Omega}
        \left(
            M_{\phi}|\mu|^2
            + M_{\psi}(\phi)|\nabla\nu|^2
        \right)
    d\mathbf{x} \le 0,
\end{equation}
ensuring thermodynamic consistency.

\subsection{MSAV Reformulation}
\label{sec:msav_reformulation}

We now introduce a reformulation based on the Multiple Scalar Auxiliary Variable (MSAV) method~\cite{cheng2018multiple,shen2018scalar}. 
Our goal is to preserve the thermodynamic structure of the original system while transforming the evolution equations into a form that leads to linear, decoupled, and energy-stable numerical schemes. The MSAV approach with multiple auxiliary variables is necessary because the disparate energy contributions, including the penalty terms for enforcing the arc-length constraint, behave very differently and cannot be properly handled with a single auxiliary variable~\cite{cheng2018multiple}.

\paragraph{Modified energy functional}
Following the MSAV approach, the nonlinear contributions from the surface, bending, area, and osmotic energies are represented using scalar auxiliary variables $V(t)$, $U(t)$, $W(t)$, and $Z(t)$, respectively. 
The standard SAV framework introduces a linear stabilization parameter $\beta$ in the energy splitting~\cite{shen2018scalar} to ensure sufficient dissipation in the implicit part of the scheme and to improve numerical stability when the nonlinear energy dominates. We reformulate the bending energy through integration by parts. 
\begin{align}
    F^{\text{bend}}[\phi]
    &=\gamma_{\text{bend}}\int_{\Omega}  \frac{3 \sqrt{2}}{16 \varepsilon} 
    \left( \frac{g'(\phi)}{\varepsilon} - \varepsilon \Delta \phi \right)^2 d\mathbf{x}\notag\\
    &=\gamma_{\text{bend}}\int_{\Omega} \frac{3 \sqrt{2}}{16 \varepsilon}
       \left(\frac{1}{\varepsilon^{2}}\left(\phi^{3}-\phi\right)^{2}
       -2|\nabla \phi|^{2}+6 \phi^{2}|\nabla \phi|^{2}
       +\varepsilon^{2}(\Delta \phi)^{2}\right) d\mathbf{x}. \label{eq:bending_ibp}
\end{align}
This decomposition reveals that the nonlinear terms 
$\frac{1}{\varepsilon^{2}}(\phi^{3}-\phi)^{2}$ and $6\phi^{2}|\nabla\phi|^{2}$ 
are non-negative, satisfying the requirement that the nonlinear part of the energy be bounded from below—a key assumption of classical SAV-type methods. The modified total energy functional is then defined as
\begin{equation} \label{eq:modified_energy}
\begin{split}
    E_{\text{mod}}(\phi, \psi; U, V, W, Z) ={}& 
    \int_{\Omega} \left[ \gamma_{\text{surf}} \frac{3 \sqrt{2}}{4}\left(\frac{\varepsilon}{2}|\nabla \phi|^{2}+\frac{\beta}{2\varepsilon}\phi^2\right) \right] d\mathbf{x} \\
    & + \int_{\Omega} \left[ \gamma_{\text{bend}} \frac{3 \sqrt{2}}{16 \varepsilon}\left(\varepsilon^{2}(\Delta \phi)^{2}-2|\nabla \phi|^{2}\right) \right] d\mathbf{x} \\
    & + \gamma_{\text{surf}} \frac{3 \sqrt{2}}{4} V^{2}
      + \frac{\gamma_{\text{area}}}{2} U^{2}
      + \frac{3 \sqrt{2}}{16\varepsilon} \gamma_{\text{bend}} W^{2}\\
      &+ \frac{1}{2}Z^2 
      - \gamma_{\text{surf}}\frac{3\sqrt{2}}{16\varepsilon}(\beta^2+2\beta),
\end{split}
\end{equation}
where the constant ensures equivalence with the original energy functional. The auxiliary variables are defined by
\begin{subequations}\label{eq:aux_vars}
\begin{align}
 V(t) &= \sqrt{\int_{\Omega} \frac{(\phi^{2}-1-\beta)^2}{4\varepsilon} \,d\mathbf{x}}, 
 & W(t) &= \sqrt{\int_{\Omega}\!\left(6 \phi^{2}|\nabla \phi|^{2}+\frac{(\phi^{3}-\phi)^{2}}{\varepsilon^{2}}\right) d\mathbf{x}},\label{eq:aux_VW}\\
 Z(t) &= \sqrt{2\int_{\Omega} f^{\text{osm}}(\phi,\psi) \, d\mathbf{x}}, 
 & U(t) &= \int_{\Omega} f^{\text{surf}}(\phi, \nabla \phi) \, d\mathbf{x} - A.\label{eq:aux_ZU}
\end{align}
\end{subequations}

\paragraph{Variational derivatives and reformulated equations}
The variational derivatives of the nonlinear energy components are
\begin{align}
    S(\phi) &= 
        \frac{\frac{1}{\varepsilon} \phi(\phi^{2}-1-\beta)}
        {2 \sqrt{\int_{\Omega} \frac{1}{4 \varepsilon} (\phi^{2}-1-\beta)^2 \,d\mathbf{x}}},\quad
    H(\phi) = 
        \frac{3 \sqrt{2}}{4}\left[\frac{1}{\varepsilon} \phi\left(\phi^{2}-1\right)-\varepsilon\Delta\phi\right], \\
    Q(\phi) &= 
        \frac{6\!\left[\phi|\nabla\phi|^2-\nabla\!\cdot(\phi^2\nabla\phi)\right]
              +\frac{1}{\varepsilon^2}(\phi^3-\phi)(3\phi^2-1)}
        {\sqrt{\int_{\Omega} \!\left( 6\phi^2|\nabla\phi|^2+\frac{1}{\varepsilon^2}(\phi^3-\phi)^2 \right) d\mathbf{x}}}, \\
    K(\phi,\psi) &= 
        \frac{p'(\phi)\!\left(f^{\text{in}}(\psi)-f^{\text{out}}(\psi)\right)}
        {2\sqrt{\int_{\Omega}\!\left[(1+p(\phi))f^{\text{in}}(\psi)+(1-p(\phi))f^{\text{out}}(\psi)\right]d\mathbf{x}}}, \\
    P(\phi,\psi) &= 
        \frac{ \gamma_{\text{in}}(1+p(\phi))(\psi-\psi_{\text{in}}) 
             + \gamma_{\text{out}}(1-p(\phi))(\psi-\psi_{\text{out}}) }
        {2\sqrt{\int_{\Omega}\!\left[(1+p(\phi))f^{\text{in}}(\psi)+(1-p(\phi))f^{\text{out}}(\psi)\right]d\mathbf{x}}}.
\end{align}
The modified governing equations are
\begin{align}
 \frac{\partial \phi}{\partial t} &= -M_{\phi} \mu, \label{eq:msav_phi_evol}\\
 \mu &= -\gamma_{\text {surf }} \frac{3 \sqrt{2} \varepsilon}{4} \Delta \phi 
        + \gamma_{\text{surf}}\frac{3\sqrt{2}\beta}{4\varepsilon}\phi
        + \gamma_{\text {bend }} \frac{3 \sqrt{2} }{4 \varepsilon} \Delta \phi 
        + \gamma_{\text {bend }} \frac{3 \sqrt{2} \varepsilon}{8} \Delta^{2} \phi \notag \\
 & \quad  + \gamma_{\text {surf }} \frac{3 \sqrt{2}}{2} V S(\phi) 
           + \gamma_{\text {area }} U H(\phi)
           + \gamma_{\text {bend }} \frac{3 \sqrt{2}}{8 \varepsilon} W Q(\phi)
           + ZK(\phi,\psi), \label{eq:msav_mu}\\ 
 \frac{\partial \psi}{\partial t} &= \nabla \cdot \left( M_{\psi}(\phi) \nabla \nu \right), \label{eq:msav_psi_evol}\\
 \nu &=  Z P(\phi,\psi).\label{eq:msav_nu}
\end{align}
These equations preserve the gradient-flow structure while transforming nonlinear terms into linear couplings through the auxiliary variables. Homogeneous Neumann conditions are imposed:
\begin{equation}
    \frac{\partial \phi}{\partial n}
    = \frac{\partial \mu}{\partial n}
    = \frac{\partial \psi}{\partial n}
    = \frac{\partial \nu}{\partial n}
    = 0
    \quad \text{on } \partial\Omega.
    \label{eq:msav_bc}
\end{equation}

\paragraph{Auxiliary variable evolution and initial conditions}
Differentiating the auxiliary-variable definitions~\eqref{eq:aux_vars} in time and applying the chain rule yields
\begin{subequations}\label{eq:aux_evolution}
\begin{align}
 \frac{dV}{dt} &=\int_{\Omega} S(\phi) \frac{\partial\phi}{\partial t} \,d\mathbf{x}, &
 \frac{dU}{dt} &=\int_{\Omega} H(\phi) \frac{\partial\phi}{\partial t} \,d\mathbf{x},\\
 \frac{dW}{dt} &=\int_{\Omega} Q(\phi) \frac{\partial\phi}{\partial t} \,d\mathbf{x}, &
 \frac{dZ}{dt} &=\int_{\Omega} \!\left[ K(\phi,\psi) \frac{\partial\phi}{\partial t}
                  + P(\phi,\psi) \frac{\partial\psi}{\partial t} \right] d\mathbf{x}.
\end{align}
\end{subequations}
Initial conditions are
\begin{align}\label{eq:msav_ic}
\phi|_{t=0} &= \phi^0, \quad \psi|_{t=0} = \psi^0, \quad
V|_{t=0} = V(\phi^0),\notag\\ U|_{t=0} &= U(\phi^0), \quad 
W|_{t=0} = W(\phi^0), \quad Z|_{t=0} = Z(\phi^0,\psi^0).
\end{align}

\paragraph{Energy dissipation of the reformulated system}
The reformulated system~\eqref{eq:msav_phi_evol}--\eqref{eq:aux_evolution} with boundary conditions~\eqref{eq:msav_bc} and initial conditions~\eqref{eq:msav_ic} is equivalent to the original system~\eqref{eq:AC}--\eqref{eq:CH} and preserves the thermodynamic structure. By taking the $L^2$ inner product of \eqref{eq:msav_phi_evol} with $\mu$, of \eqref{eq:msav_mu} with $\partial\phi/\partial t$, of \eqref{eq:msav_psi_evol} with $\nu$, of \eqref{eq:msav_nu} with $\partial\psi/\partial t$, and of~\eqref{eq:aux_evolution} with $V$, $U$, $W$, and $Z$ respectively, and then integrating by parts and summing all resulting equalities, we obtain
\begin{equation}
\label{eq:energy_dissipation}
\frac{dE_{\text{mod}}}{dt} 
= -\int_{\Omega} M_{\phi}|\mu|^2 \, d\mathbf{x} 
  - \int_{\Omega} M_{\psi}(\phi)|\nabla\nu|^2 \, d\mathbf{x} 
\leq 0.
\end{equation}
Since $M_{\phi}, M_{\psi} > 0$, the modified energy decreases monotonically, preserving thermodynamic consistency.
\section{Numerical Scheme Development}
\label{sec:numerical_scheme_development}

In this section, we develop a numerical scheme for the reformulated governing equations that achieves three main objectives: unconditional energy stability, mass conservation, and computational efficiency through constant-coefficient linear systems. The key innovation is a strategic reformulation of the stabilized MSAV approach that places Cahn--Hilliard stabilization directly in the evolution equation rather than in the chemical potential. This seemingly minor modification has profound computational consequences: all resulting elliptic systems possess constant coefficients despite the presence of variable mobility, enabling solution via fast transform methods rather than iterative solvers.

\subsection{Overview of the Numerical Approach}

Our scheme integrates three components to work together:
\begin{enumerate}
    \item \textit{Unconditional energy stability}: Rigorous discrete energy dissipation without time step restrictions (Theorems~\ref{thm:energy_stability} and~\ref{thm:fully_discrete_stability}).
    
    \item \textit{Exact mass conservation}: Preserved at the discrete level through sum-mation-by-parts identities.
    
    \item \textit{Computational efficiency}: Seven constant-coefficient elliptic systems per time step, solved directly via discrete cosine transforms.
\end{enumerate}

Sections~\ref{sec:temporal}--\ref{sec:spatial} detail the discretization and prove stability; Section~\ref{sec:computational_complexity} analyzes the computational structure and demonstrates performance gains.

\subsection{Temporal Discretization and Stabilization Strategy}
\label{sec:temporal}

The system is discretized in time using a semi-implicit, second-order backward differentiation formula (BDF2). The scheme is based on the MSAV formulation from~\cite{cheng2018multiple}, where the nonlinear energy terms handled by the auxiliary variables are treated explicitly using second-order extrapolation. 

\paragraph{Stabilization Strategy}
A key feature of our approach is the introduction of linear stabilization terms to ensure robust performance for this stiff, coupled system. We consider two stabilization terms:
\begin{itemize}
    \item \textit{Allen--Cahn stabilization} ($\theta > 0$): A standard term $\theta(\Delta\phi^{*,n+1}-\Delta\phi^{n+1})$ in the chemical potential.
    \item \textit{Cahn--Hilliard stabilization} ($\lambda > 0$): A novel term $\lambda(\Delta\psi^{n+1}-\Delta\psi^{*,n+1})$ placed directly in the evolution equation rather than in the chemical potential.
\end{itemize}

This strategic placement of the Cahn--Hilliard stabilization is crucial: while conventional approaches add stabilization inside the chemical potential~\cite{chen2019fast}, this leads to linear systems with variable coefficients due to the variable mobility $M_{\psi}(\phi)$. Our placement ensures that all resulting linear systems have constant coefficients, enabling the use of highly efficient fast direct solvers.

\paragraph{Main Time-Stepping Scheme (CC-MSAV-BDF2)}
Given  solutions $(\phi^n,  \psi^n)$ and $(\phi^{n-1}, \psi^{n-1})$ at time levels $n$ and $n-1$, we solve for the next time level $(\phi^{n+1}, \psi^{n+1})$ via the following system:
\begin{align}
    \frac{3\phi^{n+1}-4\phi^{n}+\phi^{n-1}}{2\Delta t} &= -M_{\phi} \mu^{n+1},\label{eq:AC_BDF2}\\
    \mu^{n+1} 
    &=-\gamma_{\text{surf}} \frac{3 \sqrt{2} \varepsilon}{4} \Delta \phi^{n+1}+\theta(\Delta\phi^{*,n+1}-\Delta\phi^{n+1})\notag\\
    &\quad+\gamma_{\text{surf}}\frac{3\sqrt{2}\beta}{4\varepsilon}\phi^{n+1} 
    +\gamma_{\text{bend}} \frac{3 \sqrt{2}}{4 \varepsilon} \Delta \phi^{*,n+1}\notag\\
    &\quad+\gamma_{\text{bend}} \frac{3 \sqrt{2} \varepsilon}{8} \Delta^{2} \phi^{n+1}+\gamma_{\text{surf}} \frac{3 \sqrt{2}}{2} V^{n+1} S(\phi^{*,n+1}) \notag\\
    &\quad+\gamma_{\text{area}} U^{n+1} H(\phi^{*,n+1})+\gamma_{\text{bend}} \frac{3 \sqrt{2}}{8 \varepsilon} W^{n+1} Q(\phi^{*,n+1}) \notag\\
    &\quad+Z^{n+1}K(\phi^{*,n+1},\psi^{*,n+1}),\label{eq:mu_BDF2}\\
    \frac{3\psi^{n+1}-4\psi^{n}+\psi^{n-1}}{2\Delta t}&= \lambda(\Delta\psi^{n+1}-\Delta\psi^{*,n+1})+\nabla \cdot \left( M_{\psi}(\phi^{*,n+1}) \nabla \nu^{n+1} \right),\label{eq:CH_BDF2}\\
    \nu^{n+1} &= Z^{n+1}P(\phi^{*,n+1},\psi^{*,n+1}),\label{eq:nu_BDF2}
    \end{align}
    
    \noindent
    \textit{Scalar auxiliary variable updates:}
    \begin{align}3V^{n+1}-4V^n+V^{n-1}&= \int_{\Omega} S\left(\phi^{*,n+1}\right)\left(3\phi^{n+1}-4\phi^n+\phi^{n-1}\right) \,d\mathbf{x}, \label{eq:V_BDF2}\\
    3U^{n+1}-4U^n+U^{n-1}&= \int_{\Omega} H\left(\phi^{*,n+1}\right)\left(3\phi^{n+1}-4\phi^n+\phi^{n-1}\right) \,d\mathbf{x}, \label{eq:U_BDF2}\\
    3W^{n+1}-4W^n+W^{n-1}&= \int_{\Omega} Q\left(\phi^{*,n+1}\right)\left(3\phi^{n+1}-4\phi^n+\phi^{n-1}\right) \,d\mathbf{x}, \label{eq:W_BDF2}\\
    3Z^{n+1}-4Z^n+Z^{n-1}&=\int_{\Omega} K(\phi^{*,n+1},\psi^{*,n+1})(3\phi^{n+1}-4\phi^n+\phi^{n-1})\,d\mathbf{x}\notag\\
    &\quad+\int_{\Omega} P(\phi^{*,n+1},\psi^{*,n+1})(3\psi^{n+1}-4\psi^n+\psi^{n-1})\,d\mathbf{x},\label{eq:Z_BDF2}
\end{align}
where the second-order extrapolations are defined as
\begin{equation}
\phi^{*,n+1}=2\phi^n-\phi^{n-1}, \quad \psi^{*,n+1}=2\psi^n-\psi^{n-1}.
\end{equation}

\begin{remark}[Selection of stabilization parameters]
The stabilization parameters $\theta$ and $\lambda$ are chosen to ensure physical energy stability and robust performance. Through systematic testing at $\varepsilon = 0.03125$ and $N = 256$, we adopt $\theta = 1.5$ for standard simulations (increased to $\theta = 30$ for challenging cases involving aggressive morphological changes) and $\lambda = 10^5$ for all cases, especially when the osmotic energy dominates. The detailed parameter selection procedure is provided in~\ref{app:stabilization_parameters}.
\end{remark}

\subsection{Energy Stability Analysis}

A central advantage of the MSAV framework is its ability to produce schemes that satisfy discrete energy dissipation laws, thereby inheriting the thermodynamic structure of the continuous phase-field model. For our coupled Allen--Cahn--Cahn--Hilliard system with stabilization, we establish rigorous energy stability results that depend on the stabilization parameter $\lambda$.

\paragraph{The Modified Discrete Energy}

The BDF2-MSAV framework requires a modified energy functional that accounts for both the multistep structure and the auxiliary variable formulation. For the standard MSAV approach applied to gradient flows, this modification is well-established~\cite{cheng2018multiple}. However, our strategic placement of Cahn--Hilliard stabilization (Section~\ref{sec:temporal}) introduces a novel coupling structure that necessitates careful treatment.

The modified discrete energy at time level $n+1$ takes the form
\begin{equation}
E^{n+1,n}_{\mathrm{mod}} = E^{n+1,n}_{\mathrm{BDF2}} - E^{n+1}_{\lambda},
\label{eq:modified_energy_structure}
\end{equation}
where:
\begin{itemize}
\item $E^{n+1,n}_{\mathrm{BDF2}}$ captures the phase-field, auxiliary variable, and stabilization parameter $\theta$ contributions at $t^{n+1}$,
\item $E^{n+1}_{\lambda}$ accumulates the weighted stabilization dissipation related to $\lambda$  since the second time step.
\end{itemize}

\noindent
Explicitly, the modified energy is given by
\begin{align}
E^{n+1,n}_{\mathrm{mod}} &= \frac{1}{2} \gamma_{\mathrm{bend}} \frac{3\sqrt{2} \varepsilon}{8} \left( \|\Delta \phi^{n+1}\|^2 + \|2\Delta \phi^{n+1} - \Delta \phi^n\|^2 \right) \notag \\
&\quad+ \frac{1}{2} \left( \gamma_{\mathrm{surf}} \frac{3\sqrt{2} \varepsilon}{4} + \theta \right) \left( \|\nabla \phi^{n+1}\|^2 + \|2\nabla \phi^{n+1} - \nabla \phi^n\|^2 \right) \notag \\
&\quad- \frac{1}{2} \tilde{\theta} 
\left( \|\nabla \phi^{n+1}\|^2 + \|2\nabla \phi^{n+1} - \nabla \phi^n\|^2 
- 2\|\nabla \phi^{n+1} - \nabla \phi^n\|^2 \right) \notag\\
&\quad+ \frac{1}{2} \gamma_{\mathrm{surf}} \frac{3\sqrt{2} \beta}{4\varepsilon} \left( \|\phi^{n+1}\|^2 + \|2\phi^{n+1} - \phi^n\|^2 \right) \notag \\
&\quad+ \frac{1}{2} \gamma_{\mathrm{surf}} \frac{3\sqrt{2}}{2} \left( (V^{n+1})^2 + (2V^{n+1} - V^n)^2 \right) \notag \\
&\quad+ \frac{1}{2} \gamma_{\mathrm{area}} \left( (U^{n+1})^2 + (2U^{n+1} - U^n)^2 \right) \notag \\
&\quad+ \frac{1}{2} \gamma_{\mathrm{bend}} \frac{3\sqrt{2}}{8\varepsilon} \left( (W^{n+1})^2 + (2W^{n+1} - W^n)^2 \right) \notag \\
&\quad+ \frac{1}{2}  \left( (Z^{n+1})^2 + (2Z^{n+1} - Z^n)^2 \right)\notag \\
&\quad- \sum_{j=2}^{n+1} \tfrac{\Delta t}{2}\,\|d^{j}\|^2
,
\label{eq:modified_discrete_energy}
\end{align}
where $\tilde{\theta} = \theta + \gamma_{\mathrm{bend}} \frac{3\sqrt{2}}{4\varepsilon}$ accounts for the extrapolation terms in the SAV framework, and $d^{n+1} = \frac{\lambda(\nabla\psi^{n+1}-\nabla\psi^{*,n+1})}{\sqrt{M_{\psi}(\phi^{*,n+1})}}$.

\begin{theorem}[Unconditional Energy Stability]
\label{thm:energy_stability}
Consider the numerical scheme~\eqref{eq:AC_BDF2}--\eqref{eq:Z_BDF2} with the modified discrete energy
 represents the stabilization correction. Then the following stability property holds for all $n \geq 2$:
\begin{equation}
\frac{E^{n+1,n}_{\mathrm{mod}} - E^{n,n-1}_{\mathrm{mod}}}{2\Delta t} + \frac{K^n}{2\Delta t}
= -\left[ M_{\phi} \|\mu^{n+1}\|^2 + \left\|c^{n+1}+\tfrac{1}{2}d^{n+1}\right\|^2 \right]\leq0,
\label{eq:discrete_dissipation}
\end{equation}
where $K^n \geq 0$ represents accumulated numerical dissipation arising from the BDF2 second-order differences $\|a^{n+1}-2a^n+a^{n-1}\|^2$ across all discrete variables that are not scaled by $\lambda$(see~\ref{sec:appendix_energy_stability_proof} for the explicit expression), and
\begin{equation}
c^{n+1} = \sqrt{M_{\psi}(\phi^{*,n+1})} \nabla\nu^{n+1}
\end{equation}
denotes the weighted gradient of the Cahn--Hilliard chemical potential.

\end{theorem}
\begin{proof}
The detailed proof is provided in~\ref{sec:appendix_energy_stability_proof}.
\end{proof}
\begin{remark}[Unconditional energy stability]
Inequality~\eqref{eq:discrete_dissipation} holds without restriction on $\Delta t$, $\theta$, or $\lambda$. Unlike standard MSAV schemes where stabilization enters through the chemical potential, we add it directly to the Cahn--Hilliard evolution equation, yielding (with variable mobility) the coupled dissipation $\|c^{n+1}+\tfrac{1}{2}d^{n+1}\|^2$. See~\ref{sec:appendix_energy_stability_proof} for details.
\end{remark}

\subsection{Spatial Discretization}
\label{sec:spatial}

We employ the cell-centered finite difference framework on staggered grids developed by
Wise~\cite{wise2010unconditionally}. This discretization is chosen for its
\emph{structure-preserving} properties, which are essential for transferring the
continuous energy dissipation and mass conservation laws to the fully discrete level.

In particular, the scheme satisfies:
\begin{enumerate}
    \item \textit{Discrete summation-by-parts (SBP):} exact discrete analogs of integration
    by parts, enabling rigorous energy estimates without consistency errors;
    \item \textit{Exact mass conservation:} the conservative divergence structure ensures
    $\int_\Omega \phi^{n+1}\,dx = \int_\Omega \phi^{n}\,dx$ to machine precision at every
    time step.
\end{enumerate}

The spatial domain $\Omega = (0,L_x)\times(0,L_y)$ is discretized using a uniform staggered
grid with spacing $h$. Scalar variables (e.g., $\phi,\psi$) are defined at cell
centers, while fluxes are located on cell faces. Homogeneous Neumann boundary conditions
are enforced through ghost cells. Complete definitions of the grid, discrete operators,
inner products, and SBP identities are provided in~\ref{sec:appendix_spatial_operators}.

\subsubsection{Structure-Preserving Discretization of Nonlinear Terms}

While linear operators follow standard staggered-grid constructions, special care is
required for nonlinear terms to preserve the continuous energy structure. All nonlinear
operators are discretized in conservative form so that discrete SBP identities hold. We highlight the key ingredients below.

\paragraph{Quadratic face-averaging for cubic nonlinearities}

To preserve exact discrete integration-by-parts for energy terms involving $\phi^3$, we
employ a quadratic face-averaging operator. For east--west faces,
\begin{equation}
A_x^{(q)}(\phi^2)_{i+\frac{1}{2},j}
= \tfrac{1}{3}\big(\phi_{i+1,j}^2 + \phi_{i+1,j}\phi_{i,j} + \phi_{i,j}^2\big),
\end{equation}
with an analogous definition in the $y$-direction.

This averaging preserves the discrete product rule
\[
D_x(\phi^3) = 3 A_x^{(q)}(\phi^2)\,D_x\phi,
\]
which yields the exact discrete identity
\begin{equation}
\mathbf{I}_h(\phi^3 \Delta_h \phi)
= -3\,\mathbf{I}_h(\phi^2 |\nabla_h \phi|^2),
\end{equation}
fully consistent with the continuous gradient-flow structure.

\paragraph{Variable mobility flux}

The Cahn--Hilliard flux with variable mobility is discretized conservatively as
\begin{equation}
[\nabla \cdot (M_\psi \nabla \nu)]_{i,j}
= d_x\!\left(M_\psi(A_x\phi)\,D_x\nu\right)_{i,j}
+ d_y\!\left(M_\psi(A_y\phi)\,D_y\nu\right)_{i,j},
\label{eq:mobility_flux}
\end{equation}
where the mobility is evaluated at face centers via arithmetic averaging.

\paragraph{Nonlocal bending term}

The nonlocal bending contribution
$\phi|\nabla\phi|^2 - \nabla\cdot(\phi^2\nabla\phi)$ requires a compatible discretization
of both terms. We define
\begin{align}
[\phi|\nabla\phi|^2]_{i,j}
&= \phi_{i,j}\left[
\frac{|D_x\phi|^2_{i+\frac{1}{2},j} + |D_x\phi|^2_{i-\frac{1}{2},j}}{2}
+ \frac{|D_y\phi|^2_{i,j+\frac{1}{2}} + |D_y\phi|^2_{i,j-\frac{1}{2}}}{2}
\right],\\
[\nabla\cdot(\phi^2\nabla\phi)]_{i,j}
&= d_x\!\left((A_x\phi^2)\,D_x\phi\right)_{i,j}
+ d_y\!\left((A_y\phi^2)\,D_y\phi\right)_{i,j}.
\end{align}
This conservative formulation preserves the discrete SBP structure and yields an exact
discrete bending-energy identity (see~\ref{sec:appendix_spatial_operators} for verification).

\subsubsection{Fully-Discrete Energy Stability}

The structure-preserving spatial discretization enables a direct extension of the
semi-discrete energy stability result to the fully discrete scheme.

\begin{theorem}[Fully-Discrete Energy Stability]
\label{thm:fully_discrete_stability}
Let $E_h^{n+1,n}$ denote the fully discrete modified energy obtained from
\eqref{eq:modified_discrete_energy} by replacing all continuous operators and norms with
their discrete counterparts defined in~\ref{sec:appendix_spatial_operators}. Then, for all $\Delta t>0$ and $h>0$,
\begin{equation}
\frac{E_h^{n+1,n} - E_h^{n,n-1}}{2\Delta t}
\leq
-\left[
M_\phi \|\mu_h^{n+1}\|_2^2
+ \left\|c_h^{n+1} + \tfrac{1}{2}d_h^{n+1}\right\|_2^2
\right]
\le 0.
\end{equation}
\end{theorem}

\begin{proof}
The proof follows identically to Theorem~\ref{thm:energy_stability}, with continuous inner
products replaced by their discrete analogs. Homogeneous Neumann boundary conditions
ensure that all boundary terms vanish in the discrete SBP identities, and the discrete
Green's identity $(\phi,\Delta_h\psi)_h = -\|\nabla_h\phi\|_2^2$ holds exactly.
\end{proof}

\begin{remark}[Mass conservation]
Discrete mass conservation follows directly from the conservative flux formulation
\eqref{eq:mobility_flux}. For any discrete flux field satisfying the boundary conditions,
\[
h^2\sum_{i,j} [\nabla_h\cdot\mathbf{F}]_{i,j} = 0,
\]
implying $\sum_{i,j}\psi_{i,j}^{n+1}=\sum_{i,j}\psi_{i,j}^{n}$.
Numerical experiments in Section~\ref{sec:numerical_results} confirm relative mass errors
below $5\times 10^{-11}$.
\end{remark}

\subsection{Computational Efficiency: Fast Direct Solution}
\label{sec:computational_complexity}

A primary objective of this work is to improve the computational efficiency of 
phase-field vesicle simulations through constant-coefficient reformulation. The 
essential distinction between the proposed and classical MSAV approaches lies in the 
structure of the resulting linear systems: the classical formulation produces 
\emph{variable-coefficient} elliptic equations in the Cahn--Hilliard subsystem due to 
spatially-varying mobility, while the proposed reformulation yields 
\emph{constant-coefficient} elliptic operators for all linear subproblems. This 
structural property enables direct solution via discrete cosine transforms (DCT) 
without matrix assembly or iteration, in contrast to classical methods requiring 
either sparse direct solvers or preconditioned Krylov iterations.

To provide a controlled comparison isolating the algorithmic contribution, we compare 
both MSAV implementations in \texttt{MATLAB} R2025b using identical numerical 
infrastructure (built-in sparse direct solver and DCT functions). Each time step 
requires solving seven linear elliptic systems: five fourth-order equations from the Allen--Cahn subsystem and two second-order 
equations from the Cahn--Hilliard subsystem. For rectangular domains with 
homogeneous Neumann boundary conditions, the constant-coefficient structure of all 
seven operators admits efficient DCT-based solutions.

\subsubsection{Solution Algorithm and Complexity}

The solution procedure at each time step consists of four stages with distinct 
computational characteristics:

\paragraph{Stages 1, 3--4: Lower-order operations}
Explicit extrapolations $\phi^{*,n+1} = 2\phi^n - \phi^{n-1}$ and 
$\psi^{*,n+1} = 2\psi^n - \psi^{n-1}$ (Stage~1), solution of a dense $5 \times 5$ 
algebraic system for scalar auxiliary variable coefficients (Stage~3), and 
reconstruction of updated fields as linear combinations (Stage~4) each require 
$O(N^2)$ operations and contribute negligibly to overall complexity.

\paragraph{Stage 2: Elliptic system solution (dominant cost)}
The implicit discretization yields constant-coefficient operators
\begin{align}
    \chi(\phi) &= C_{1}\,\phi - C_{2}\,\Delta_h \phi + C_{3}\,\Delta_h^2 \phi, \label{eq:allen_cahn_op}\\
    \zeta(\psi) &= \left( \frac{3}{2\Delta t} - \lambda \Delta_h \right)\psi, \label{eq:cahn_hilliard_op}
\end{align}
where $C_{1}, C_{2}, C_{3} > 0$ depend on model parameters and $\Delta_h$ denotes the 
discrete Laplacian. The five fourth-order equations share operator~\eqref{eq:allen_cahn_op}, 
while the two second-order equations share operator~\eqref{eq:cahn_hilliard_op}. The 
DCT diagonalizes both operators, reducing each solve to pointwise division in 
transform space (see~\ref{sec:appendix_fast_direct_solver} for implementation details).

The natural partitioning into two operator groups enables \emph{batched} DCT 
evaluation: right-hand sides sharing identical operators are stacked along a third 
array dimension and transformed simultaneously via \texttt{MATLAB}'s multidimensional 
DCT. This requires only two forward-inverse DCT pairs per time step (one per operator 
type) rather than individual transforms for each of the seven equations. Each DCT 
evaluation requires approximately $4N^2\log_2 N$ floating-point operations, yielding 
dominant per-timestep complexity
\[
    \mathcal{C}_{\mathrm{DCT}} = O(N^2 \log N),
\]
with a modest constant factor reflecting the specific structure (seven elliptic solves 
grouped into two batched operations). Beyond asymptotic scaling, batching reduces 
function call overhead and improves cache efficiency, contributing to the 
observed performance gains in Section~\ref{sec:performance_results}.

\subsubsection{Baseline Methods}
\label{sec:baseline_methods}

For comparison, we implement the classical MSAV formulation using two 
standard approaches for the variable-coefficient Cahn--Hilliard subsystem:

\paragraph{Sparse direct solver (Classical+Direct)}
Assemble one sparse variable-coefficient matrix per time step for two direct solves 
(via \texttt{MATLAB}'s backslash) of the Cahn--Hilliard subsystem equations. Robust but 
exhibits steep computational and memory growth with grid refinement. We also tested reuse of sparse LDL factorizations for the two Cahn–Hilliard solves per time step; however, at large grid sizes, the factorization cost and memory traffic dominated, providing no net speedup. We therefore report results using \texttt{MATLAB}’s optimized sparse backslash, which delivered the best direct-solver performance in our tests.

\paragraph{DCT-preconditioned PCG (Classical+DCT-PCG)}
Matrix-free operator application with DCT-based Poisson preconditioner (spatial-average 
mobility), constructed once per time step for two PCG solves. Maintains mesh-independent 
iteration counts (~12--13 iterations) without parameter tuning. Incomplete 
factorizations (ILU, IC) performed poorly and are excluded.

The Allen--Cahn subsystem yields constant-coefficient equations in all formulations, 
solved identically via DCT across all three implementations. Performance differences 
therefore isolate the Cahn--Hilliard subsystem treatment, where the proposed 
constant-coefficient reformulation enables direct DCT solution versus variable-coefficient 
equations requiring sparse direct or iterative solvers in classical methods. 

While reference implementations using fully coupled nonlinear multigrid (NLMG) in 
\texttt{Fortran} exist in the literature, our \texttt{MATLAB}-based comparison with identical 
numerical infrastructure provides a controlled assessment of the algorithmic impact of the constant-coefficient 
reformulation.
\subsubsection{Performance Results}
\label{sec:performance_results}

Figure~\ref{fig:runtime_comparison} demonstrates substantial efficiency gains from the 
constant-coefficient reformulation across grid sizes $N = 128$ to $4096$. At $N=4096$, 
the proposed CC-MSAV+DCT scheme achieves approximately 6.5$\times$ speedup over 
DCT-preconditioned PCG and 15$\times$ over sparse direct methods. Both formulations discretize the same continuous model and maintain second-order accuracy in time and space, unconditional energy stability, and produce nearly indistinguishable solutions. 
(cf. Section~\ref{sec:efficiency_accuracy}) This confirms that efficiency gains arise 
purely from algorithmic structure rather than accuracy trade-offs.

\begin{figure}[H]
    \centering
    \includegraphics[width=.95\linewidth]{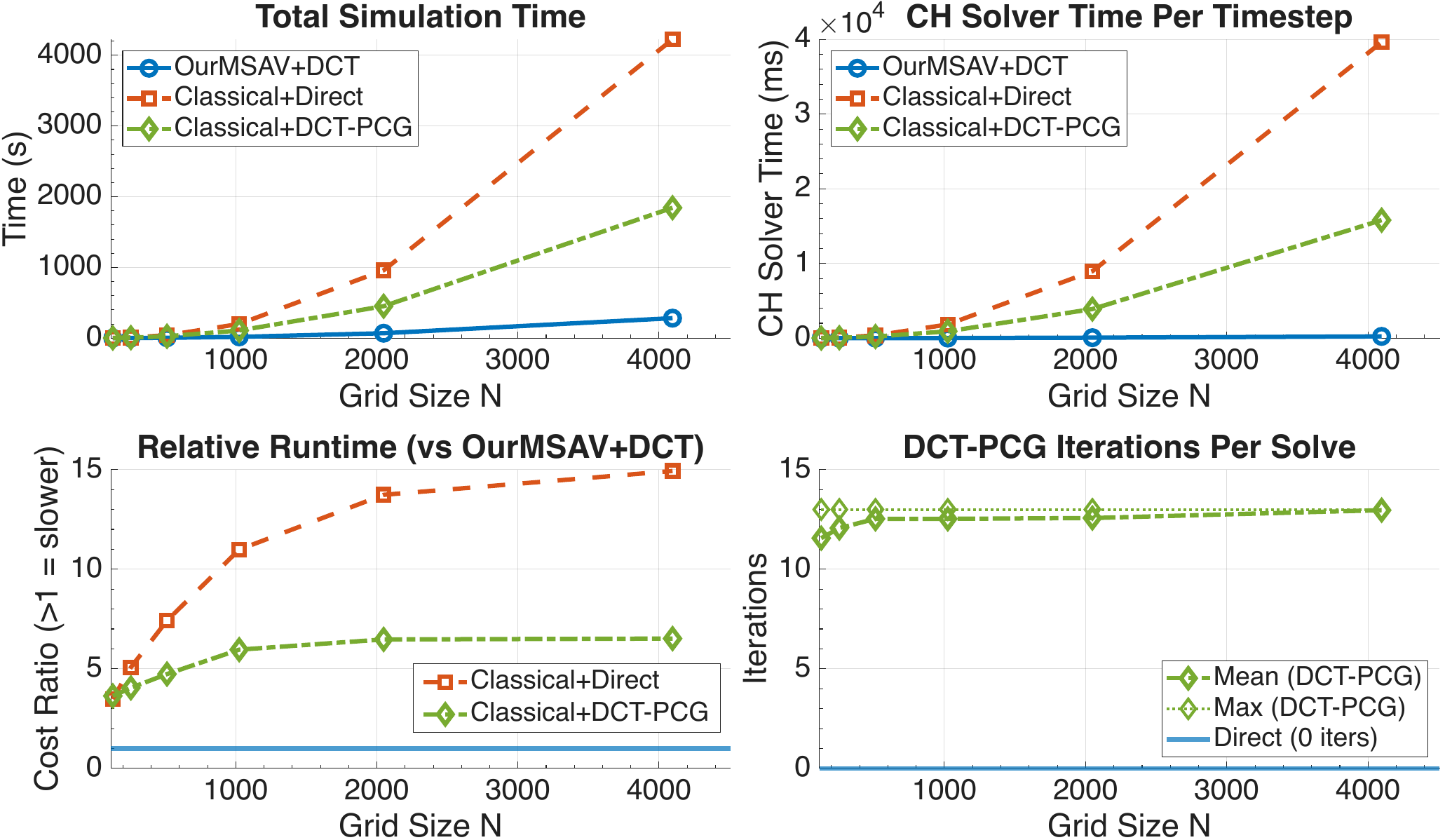}
    \caption{
Runtime comparison for proposed CC-MSAV+DCT (blue) versus classical MSAV with 
sparse direct solver (red) and DCT-preconditioned PCG (green) on grids 
$N=128$--$4096$. \textbf{(Top left)} Total simulation time for 100 time steps. 
\textbf{(Top right)} Per-timestep solver cost for Cahn--Hilliard subsystem. 
\textbf{(Bottom left)} Speedup factors showing 6--15$\times$ acceleration at 
$N=4096$. \textbf{(Bottom right)} PCG iteration counts (green diamonds) remain 
mesh-independent at 12--13 iterations; blue line marks zero iterations for direct DCT. 
All timings in \texttt{MATLAB} R2025b on identical hardware.
}
\label{fig:runtime_comparison}
\end{figure}
To further isolate the dominant computational bottleneck at the largest tested resolution, Table~\ref{tab:ch_cost_4096} reports the per-timestep Cahn--Hilliard solver cost and its fraction of total runtime at $N=4096$.
\begin{table}[H]
\centering
\caption{Cahn--Hilliard solver cost and overhead at $N=4096$ (100 time steps).}
\label{tab:ch_cost_4096}
\resizebox{\linewidth}{!}{%
\begin{tabular}{lccc}
\hline
Method & CH time per step (s) & CH fraction of total (\%) & Relative CH speedup \\
\hline
CC-MSAV+DCT        & 0.237  & 8.4  & 1.0$\times$ \\
Classical+Direct  & 39.63  & 93.8 & 167$\times$ \\
Classical+DCT-PCG & 15.80  & 85.7 & 67$\times$ \\
\hline
\end{tabular}
}
\end{table}
\paragraph{Scaling analysis}
The key panels reveal distinct scaling behaviors:
\begin{itemize}
\item \textit{Total runtime} (top left): The proposed method exhibits substantially 
milder growth with grid refinement compared to classical approaches. At $N=4096$, 
100 time steps require approximately $283$~s for CC-MSAV+DCT versus $1844$~s for 
Classical+DCT-PCG and $4224$~s for Classical+Direct.

\item \textit{Per-timestep Cahn--Hilliard solver cost} (top right): This panel isolates 
the core algorithmic difference. At $N=4096$, the per-timestep Cahn--Hilliard cost is 
$0.237$~s for CC-MSAV+DCT, compared with $15.8$~s for Classical+DCT-PCG 
($\sim67\times$ slower) and $39.6$~s for Classical+Direct ($\sim167\times$ slower). 
The Cahn--Hilliard subsystem dominates the runtime of classical formulations but 
accounts for less than $10\%$ of the total cost in the proposed scheme.

\item \textit{Speedup factors} (bottom left): The advantage increases monotonically 
with grid refinement, from approximately $3$--$5\times$ at $N=256$ to $6$--$15\times$ 
at $N=4096$. This trend reflects the fundamental structural benefit of the proposed 
approach: classical methods require either repeated sparse matrix assembly or 
multiple Krylov iterations per time step, whereas the proposed scheme eliminates both 
through direct batched DCT solvers with $O(N^2\log N)$ complexity.

\item \textit{Iteration counts} (bottom right): The DCT-based Poisson preconditioner 
maintains mesh-independent performance, requiring $12$--$13$ PCG iterations per solve 
across all grid sizes. The proposed method requires zero iterations by design.
\end{itemize}

\paragraph{Practical implications}
These efficiency gains enable several computational advantages for vesicle dynamics 
studies:
\begin{enumerate}
\item \textit{High-resolution long-time simulations}: Resolutions $N \geq 2048$ become 
feasible for extended time integration, critical for capturing fine interfacial 
structures and long-time morphological evolution.

\item \textit{Rapid parameter exploration}: Order-of-magnitude speedups facilitate 
systematic parameter sweeps, enabling more thorough investigation of model behavior 
across physically relevant regimes.

\item \textit{Complex multi-vesicle systems}: Computational savings can be allocated to 
increased spatial detail or additional vesicles, expanding the scope of accessible 
biological scenarios.
\end{enumerate}

The combination of unconditional energy stability, second-order accuracy, $O(N^2 \log N)$ 
complexity per time step, and low memory footprint (no assembled matrices) makes the 
proposed scheme particularly well-suited for large-scale computational investigations 
of biological membrane dynamics.

\paragraph{Limitations and scope}
The constant-coefficient reformulation achieves maximum efficiency on rectangular 
domains with separable homogeneous boundary conditions, enabling fast transform 
solution via DCT (Neumann, as implemented), DST (Dirichlet), or FFT (periodic), all 
with $O(N^2 \log N)$ complexity. For irregular geometries or inhomogeneous boundary 
conditions, the method requires adaptation (e.g., domain embedding or preconditioned iterative solvers), potentially reducing but not 
eliminating the performance advantage. The fully decoupled structure remains 
advantageous for implementation simplicity and parallel scalability even when 
specialized iterative methods become competitive on complex geometries. The 
embarrassingly parallel structure of batched transform operations also suggests 
significant potential for GPU acceleration, which remains a subject for future 
investigation.
\section{Numerical Results and Validation}
\label{sec:numerical_results}

In this section, we present numerical experiments demonstrating the accuracy and robustness of the proposed MSAV-based scheme. We begin by verifying the theoretical convergence rates through temporal, spatial, and coupled space–time studies (Section~\ref{sec:convergence_studies}), then validating physical accuracy against the established NLMG benchmark (Section~\ref{sec:benchmark}). Having confirmed correctness, we demonstrate that the constant-coefficient reformulation achieves significant computational efficiency without compromising accuracy (Section~\ref{sec:efficiency_accuracy}), and finally show solver robustness in geometrically complex single-vesicle scenarios (Section~\ref{sec:complex_dynamics}).

\subsection{Convergence Studies}
\label{sec:convergence_studies}

We verify the theoretical accuracy of the MSAV-BDF2 scheme through temporal 
and spatial convergence studies using the smooth elliptical initial condition 
from~\cite{tang2023phase}. All parameters are listed in 
Table~\ref{tab:benchmark_params} (\ref{sec:appendix_simulation_parameters}).

\paragraph{Temporal Convergence}

Temporal convergence is measured using the Cauchy error method, comparing 
solutions at time steps $\Delta t$ and $\Delta t/2$ at final time 
$T = 5 \times 10^{-3}$. We define the combined error for the coupled system as
\begin{equation}
\mathcal{E}_{\Delta t} = \sqrt{\|\phi^{\Delta t} - \phi^{\Delta t/2}\|_{L^2}^2 + \|\psi^{\Delta t} - \psi^{\Delta t/2}\|_{L^2}^2},
\end{equation}
with convergence rate $p = \log_2(\mathcal{E}_{\Delta t}/\mathcal{E}_{\Delta t/2})$. 
Tests use moderate stiffness $(\gamma_{\text{in}}, \gamma_{\text{out}}, \gamma_{\text{area}}) = (100, 100, 100)$ 
and osmotic equilibrium $\psi_{\text{in}} = 0.1$ on a fixed $512^2$ grid. Table~\ref{tab:temporal_convergence_cauchy} shows second-order convergence 
with combined rates of 2.08--2.73 (average 2.40), confirming BDF2 accuracy.

\begin{table}[H]
    \centering
    \caption{Temporal convergence rates for the coupled system at $T = 5 \times 10^{-3}$.}
    \label{tab:temporal_convergence_cauchy}
    \begin{tabular}{|c|c|c|c|c|c|}
        \hline
        \textbf{$\Delta t$} & \multicolumn{2}{c|}{\textbf{$\phi$}} & 
        \multicolumn{2}{c|}{\textbf{$\psi$}} & \textbf{Combined} \\
        \cline{2-6}
        & Error & Rate & Error & Rate & Rate \\
        \hline
        $1.0 \times 10^{-5}$ & $9.98 \times 10^{-5}$ & --- & $1.26 \times 10^{-2}$ & --- & --- \\
        $5.0 \times 10^{-6}$ & $1.26 \times 10^{-5}$ & 2.98 & $2.99 \times 10^{-3}$ & 2.08 & 2.08 \\
        $2.5 \times 10^{-6}$ & $8.82 \times 10^{-7}$ & 3.84 & $4.52 \times 10^{-4}$ & 2.73 & 2.73 \\
        $1.25 \times 10^{-6}$ & $7.81 \times 10^{-8}$ & 3.50 & $8.70 \times 10^{-5}$ & 2.38 & 2.38 \\
        \hline
    \end{tabular}
\end{table}

\paragraph{Spatial Convergence}

Spatial convergence is measured via numerical error between successive grid 
refinements $N \in \{128, 256, 512, 1024, 2048\}$ with fixed $\Delta t = 10^{-6}$ 
to ensure negligible temporal error. The error  defined as
$
\mathcal{E}_h = \|u^h - \mathcal{I}(u^{h/2})\|_{L^2},
$
where $\mathcal{I}$ denotes cubic interpolation to the coarser grid. The combined 
error is defined analogously. Table~\ref{tab:spatial_convergence} and 
Figure~\ref{fig:spatial_convergence} demonstrate second-order spatial accuracy 
with asymptotic rate 2.00 for both fields and the coupled system.

\begin{table}[H]
    \centering
    \caption{Spatial convergence rates at $T = 5 \times 10^{-3}$.}
    \label{tab:spatial_convergence}
    \begin{tabular}{|c|c|c|c|c|}
        \hline
        \textbf{$N$} & \textbf{$\phi$ Error} & \textbf{Rate} &
        \textbf{$\psi$ Error} & \textbf{Rate} \\
        \hline
        128  & $3.7852 \times 10^{-4}$ & ---  & $3.0688 \times 10^{-5}$ & ---  \\
        256  & $9.5905 \times 10^{-5}$ & 1.98 & $7.7845 \times 10^{-6}$ & 1.98 \\
        512  & $2.4057 \times 10^{-5}$ & 2.00 & $1.9532 \times 10^{-6}$ & 1.99 \\
        1024 & $6.0191 \times 10^{-6}$ & 2.00 & $4.8875 \times 10^{-7}$ & 2.00 \\
        \hline
    \end{tabular}
\end{table}

\begin{figure}[H]
    \centering
    \includegraphics[width=0.80\textwidth]{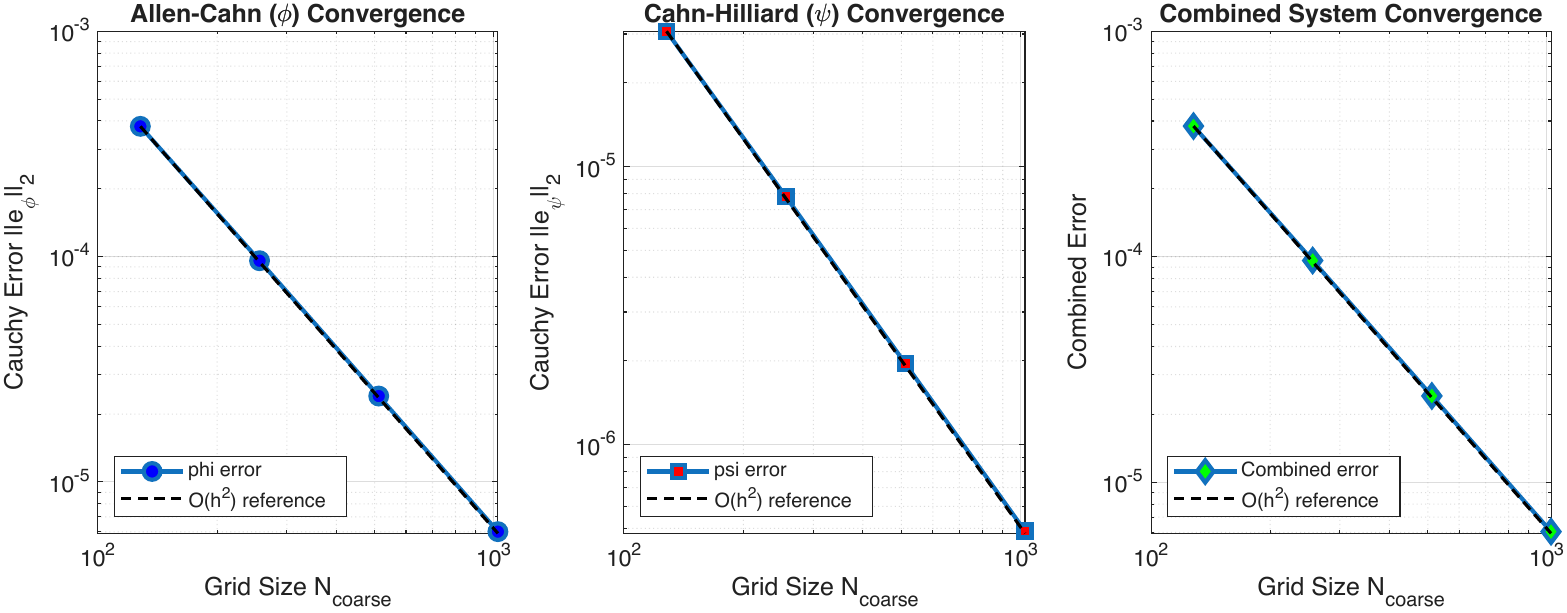}
    \caption{Spatial convergence for $\phi$ (left), $\psi$ (middle), and combined 
    system (right). Dashed lines show theoretical $O(h^2)$ rate.}
    \label{fig:spatial_convergence}
\end{figure}

\paragraph{Refinement-Path Convergence}
\label{sec:refinement_convergence}

To assess the overall convergence of the fully coupled scheme, we perform a
refinement-path test where spatial and temporal resolutions are refined
simultaneously according to $\Delta t = C h$ with $C = 0.002$. The
configuration parameters are listed in Table~\ref{tab:convergence_results}, with grid
sizes from $N_0 = 128$ to $N_{\max} = 2048$ and final time $T = 5 \times
10^{-3}$. The error between successive refinement levels is defined as
\begin{equation}
\mathcal{E}_{h,\Delta t} =
\sqrt{\|\phi^{(N)} - \mathcal{I}(\phi^{(2N)})\|_{L^2}^2 +
      \|\psi^{(N)} - \mathcal{I}(\psi^{(2N)})\|_{L^2}^2},
\end{equation}
where $\mathcal{I}$ denotes cubic interpolation of the finer solution to the
coarser grid. 

Table~\ref{tab:convergence_results} summarizes the coupled space–time
convergence results. The scheme exhibits consistent second-order accuracy for
both $\phi$ and $\psi$, with combined rates between 1.83 and 2.53 (average rate
$\approx 2.2$). These results confirm that the MSAV–BDF2 scheme maintains
global second-order accuracy under simultaneous space–time refinement.

\begin{table}[H]
\centering
\caption{Coupled space--time convergence results.}
\label{tab:convergence_results}
\resizebox{\textwidth}{!}{%
\begin{tabular}{ccccccccc}
\toprule
\textbf{Level} & $\boldsymbol{N}$ & $\boldsymbol{\Delta t}$ & $\boldsymbol{\|\mathcal{E}_\phi\|}$ & $\boldsymbol{r_\phi}$ & $\boldsymbol{\|\mathcal{E}_\psi\|}$ & $\boldsymbol{r_\psi}$ & $\boldsymbol{\|\mathcal{E}_{\text{comb}}\|}$ & $\boldsymbol{r_{\text{comb}}}$ \\
\midrule
1 & 128  & $1.563 \times 10^{-5}$ & $2.3757 \times 10^{-3}$ & ---   & $2.7560 \times 10^{-2}$ & ---   & $2.7662 \times 10^{-2}$ & --- \\
2 & 256  & $7.813 \times 10^{-6}$ & $5.8906 \times 10^{-4}$ & 2.012 & $5.2819 \times 10^{-3}$ & 2.383 & $5.3147 \times 10^{-3}$ & 2.381 \\
3 & 512  & $3.906 \times 10^{-6}$ & $1.4080 \times 10^{-4}$ & 2.065 & $1.4821 \times 10^{-3}$ & 1.833 & $1.4887 \times 10^{-3}$ & 1.835 \\
4 & 1024 & $1.953 \times 10^{-6}$ & $3.6117 \times 10^{-5}$ & 1.963 & $2.5566 \times 10^{-4}$ & 2.535 & $2.5820 \times 10^{-4}$ & 2.530 \\
\bottomrule
\end{tabular}
}
\end{table}

\subsection{Benchmark Validation}
\label{sec:benchmark}

We validate physical accuracy by comparing against the NLMG baseline 
solver~\cite{tang2023phase} for canonical osmotic growth and shrinkage scenarios. 
Both tests use the same elliptical initial condition for $\phi^0$ with distinct 
concentration profiles: $\psi^0 = -0.35\phi^0 + 0.45$ (growth) and 
$\psi^0 = -0.1\phi^0 + 0.7$ (shrinkage). All parameters follow 
Table~\ref{tab:benchmark_params} with $\varepsilon = 0.03125$, $N = 256$, 
$\Delta t = 10^{-6}$, providing $\sim$ 8 grid points per interface width.

Figures~\ref{fig:benchmark_growth}--\ref{fig:benchmark_shrinkage} demonstrate 
excellent agreement between both methods across all physical quantities: 
interface morphology, energy dissipation, arc length conservation, mass 
conservation, and concentration evolution. Both solvers correctly 
capture the multiphysics coupling and preserve thermodynamic consistency, validating the proposed MSAV scheme.

\begin{figure}[H]
    \centering
    \includegraphics[width=1.0\linewidth]{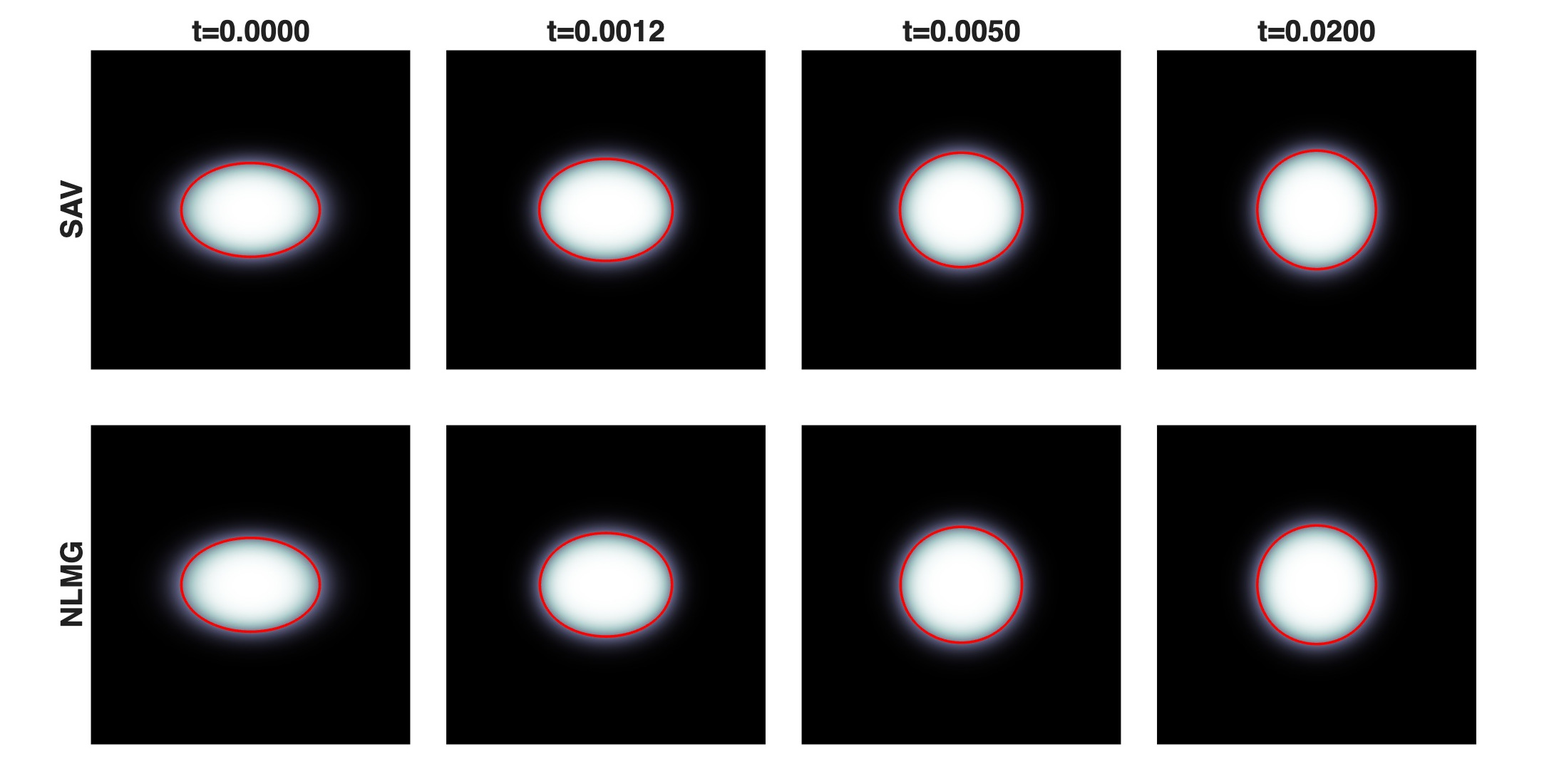}
    \includegraphics[width=0.93\linewidth]{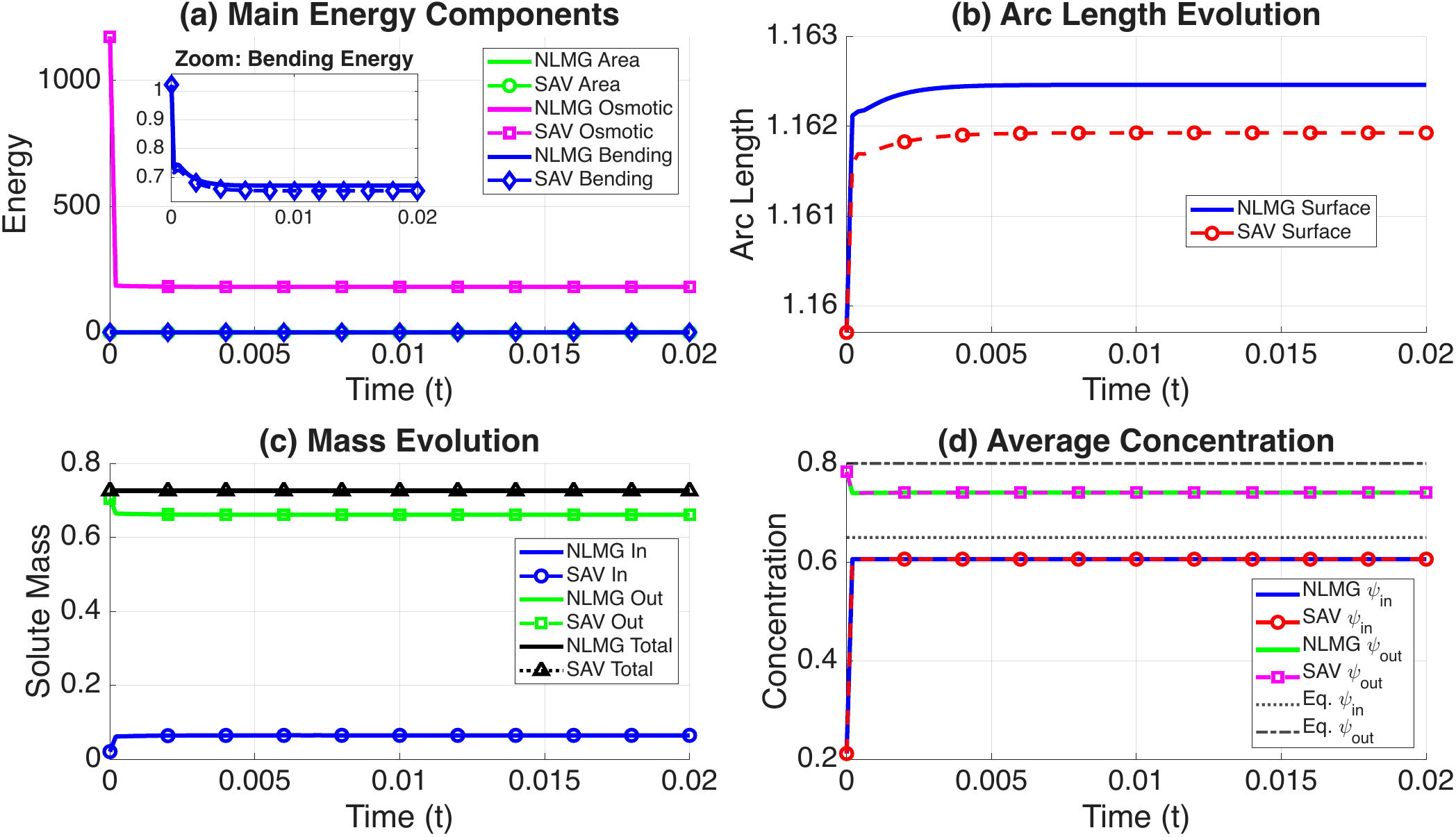}
    \caption{Elliptical growth validation. Top: Interface evolution (SAV: row 1; 
    NLMG: row 2) at $t = 0, 0.005, 0.01, 0.02$. Bottom: Energy components, arc 
    length, mass, and concentration dynamics. Excellent agreement validates 
    physical accuracy.}
    \label{fig:benchmark_growth}
\end{figure}

\begin{figure}[H]
    \centering
    \includegraphics[width=1.0\linewidth]{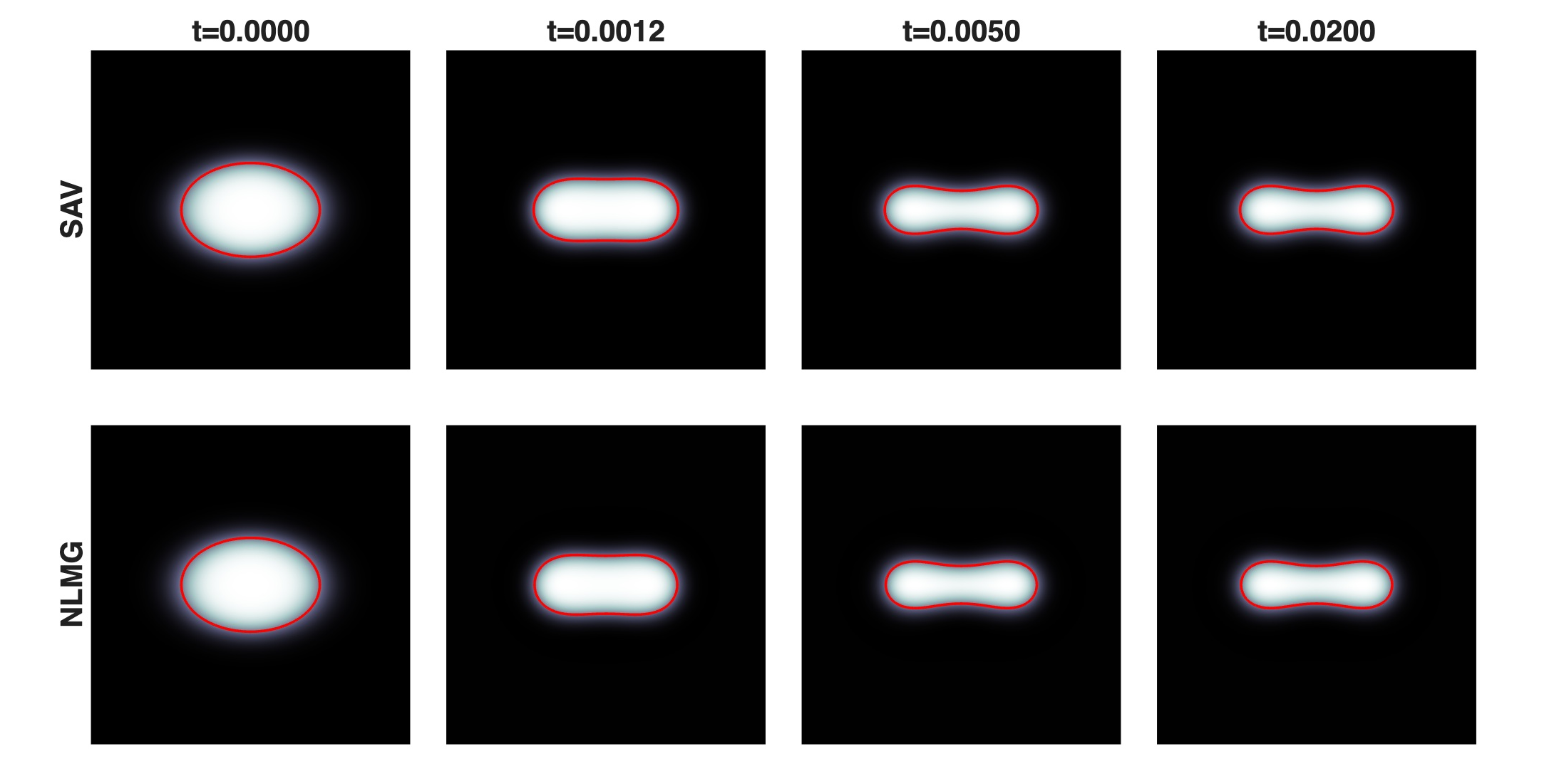}
    \includegraphics[width=0.93\linewidth]{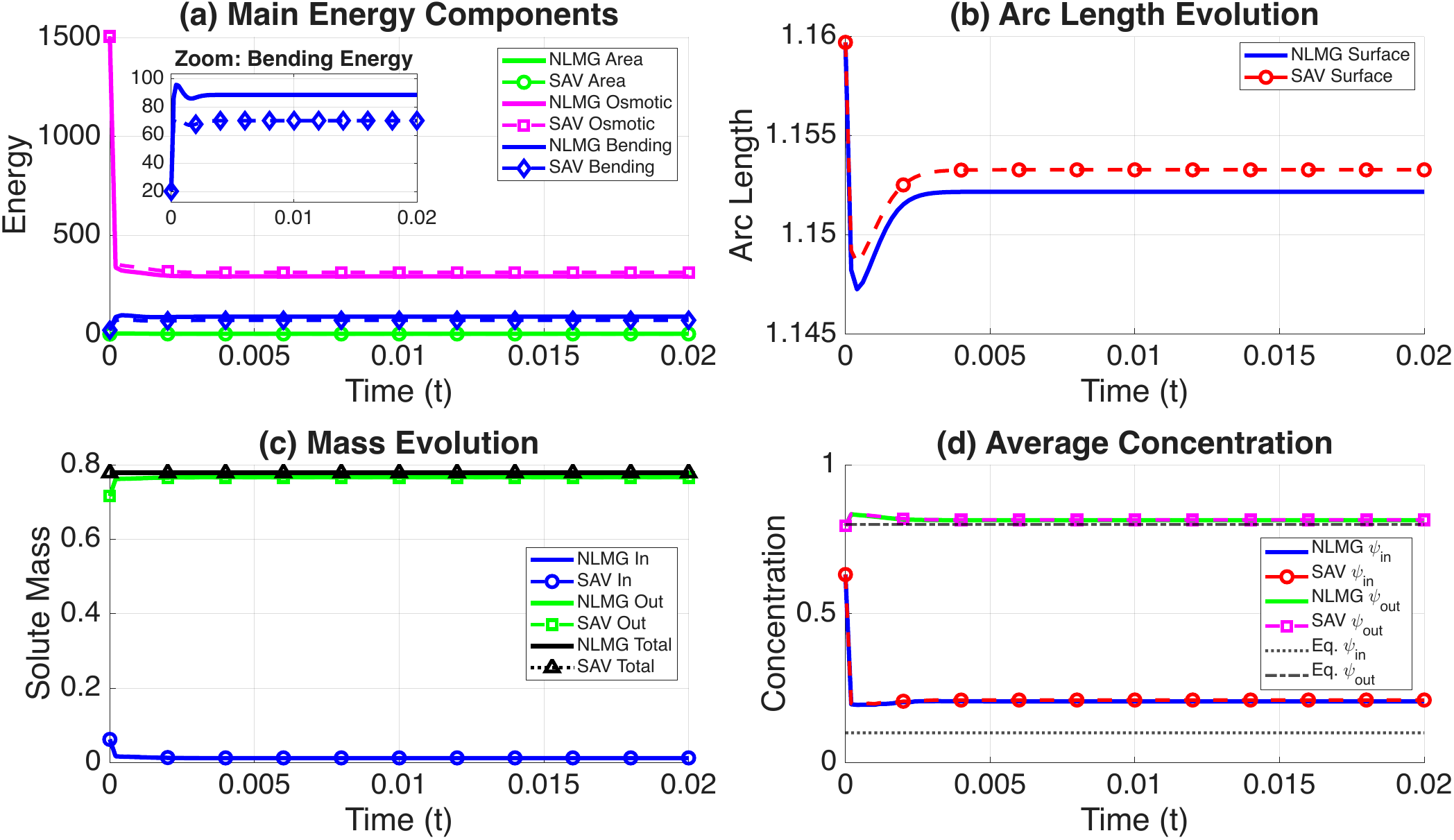}
    \caption{Elliptical shrinkage validation with $\gamma_{\text{bend}} = 1.0$ at 
    $t = 0, 0.0012, 0.005, 0.02$. Top: Morphology (SAV: row 1; NLMG: row 2). 
    Bottom: Quantitative comparison showing consistent dissipation, conservation, 
    and dynamics.}
    \label{fig:benchmark_shrinkage}
\end{figure}
\newpage
\subsection{Efficiency--Accuracy Validation}
\label{sec:efficiency_accuracy}
 Having validated both  convergence (Section~\ref{sec:convergence_studies}) 
and physical accuracy (Section~\ref{sec:benchmark}), we now examine whether the computational efficiency gains reported in Section~\ref{sec:computational_complexity} are obtained without any loss of accuracy. The validation is performed using the elliptical growth test described in Section~\ref{sec:benchmark}.

\begin{figure}[H]
    \centering
    \includegraphics[width=0.9\linewidth]{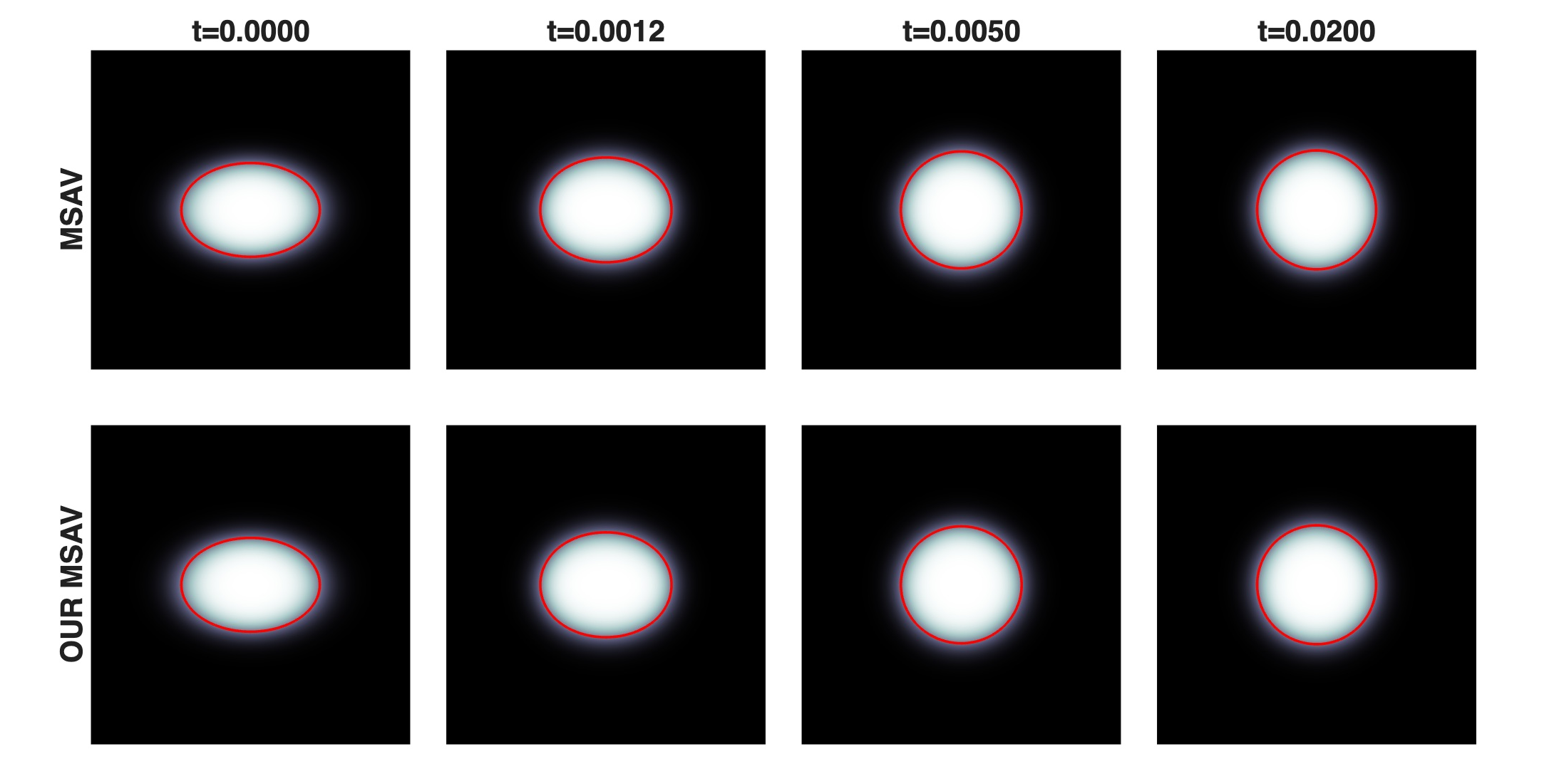}

    \includegraphics[width=0.9 \linewidth]{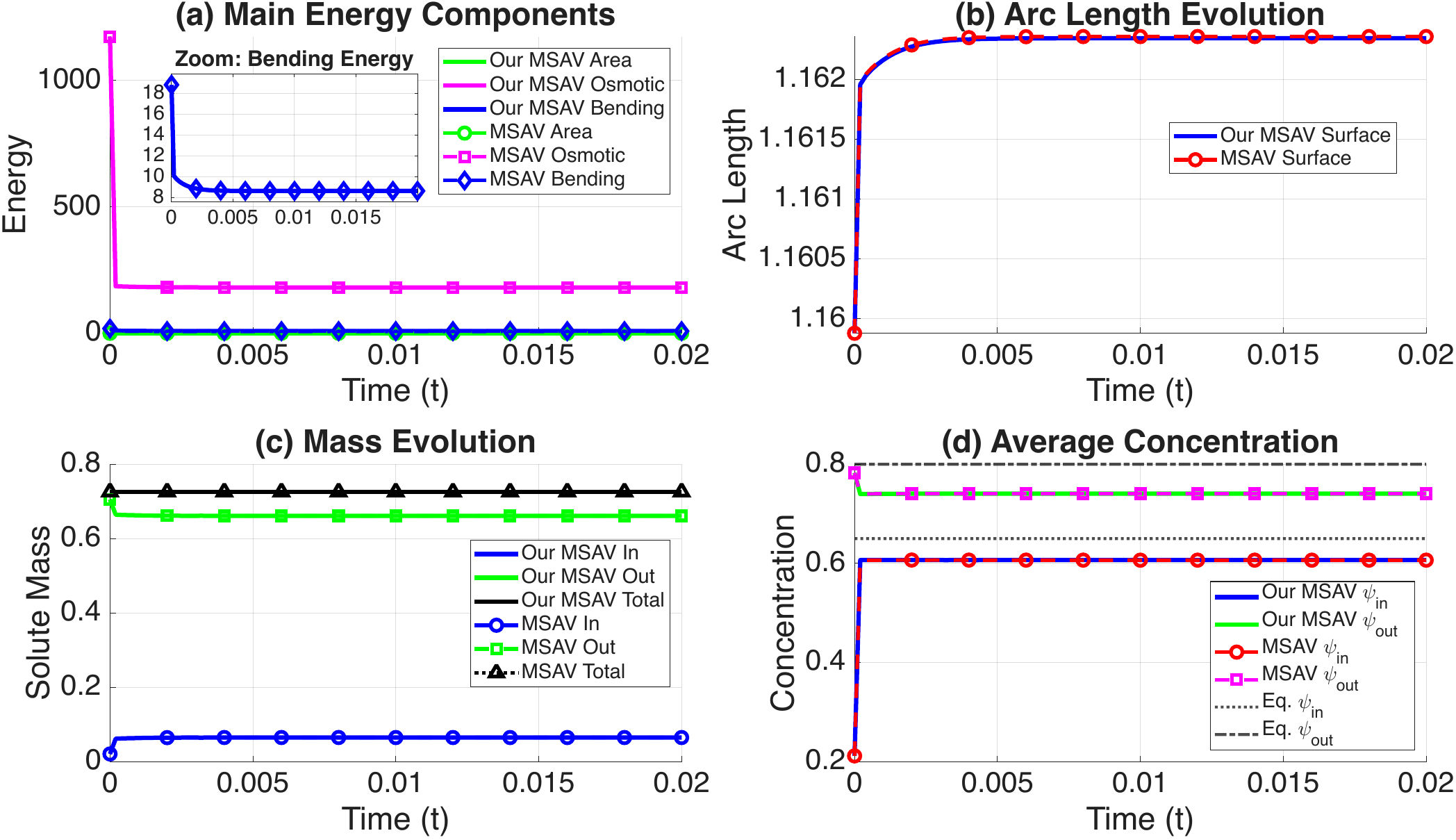}
    \caption{Accuracy comparison for the elliptical growth test. The first row corresponds to the classical MSAV method and the second row corresponds to the proposed CC-MSAV method. The top panels show the interface evolution at $t = 0$, $0.0012$, $0.005$, and $0.02$. The bottom panels contain the energy components and conservation metrics. The results remain visually and quantitatively indistinguishable, which confirms that the efficiency gains reported in Figure~\ref{fig:runtime_comparison} are obtained without a reduction in accuracy.}
    \label{fig:msav_comparison}
\end{figure}

\subsection{Robustness to Complex Geometries and Large Deformations}
\label{sec:complex_dynamics}

Having validated accuracy and efficiency on smooth test problems 
(Sections~\ref{sec:convergence_studies}--\ref{sec:efficiency_accuracy}), 
we now examine solver robustness under geometrically challenging conditions. A critical challenge for phase-field methods in vesicle dynamics is maintaining stability and conservation properties under large geometric deformations involving extreme morphologies. We evaluate the robustness of the proposed MSAV scheme through single-vesicle simulations with geometrically complex initial conditions that stress-test the method's ability to handle sharp corners, concave regions, and high local curvature. All simulations employ a well-resolved diffuse-interface representation with
$
\varepsilon = 0.01, \quad N = 2048,  \quad \Delta t = 5\times 10^{-6},
$
providing approximately twenty grid points across the interface. Non-elliptical initial shapes are smoothed using the procedure described in~\ref{sec:appendix_smoothing}. Unless otherwise specified, physical parameters follow Table~\ref{tab:benchmark_params}.
\vspace{0.3cm}\\
\textit{Growth dynamics.} 
We first consider volume expansion under hypertonic conditions with parameters $(\gamma_{\mathrm{bend}},\theta) = (0.05,10)$. Two distinct initial morphologies (triangle and star) are tested. As shown in Figure~\ref{fig:single_vesicle_gr_robustness}, both configurations undergo substantial rounding and evolve towards circular equilibrium configurations, consistent with energy minimization for expanding vesicles.

\begin{figure}[H]
    \centering
    \includegraphics[width=0.80\linewidth]{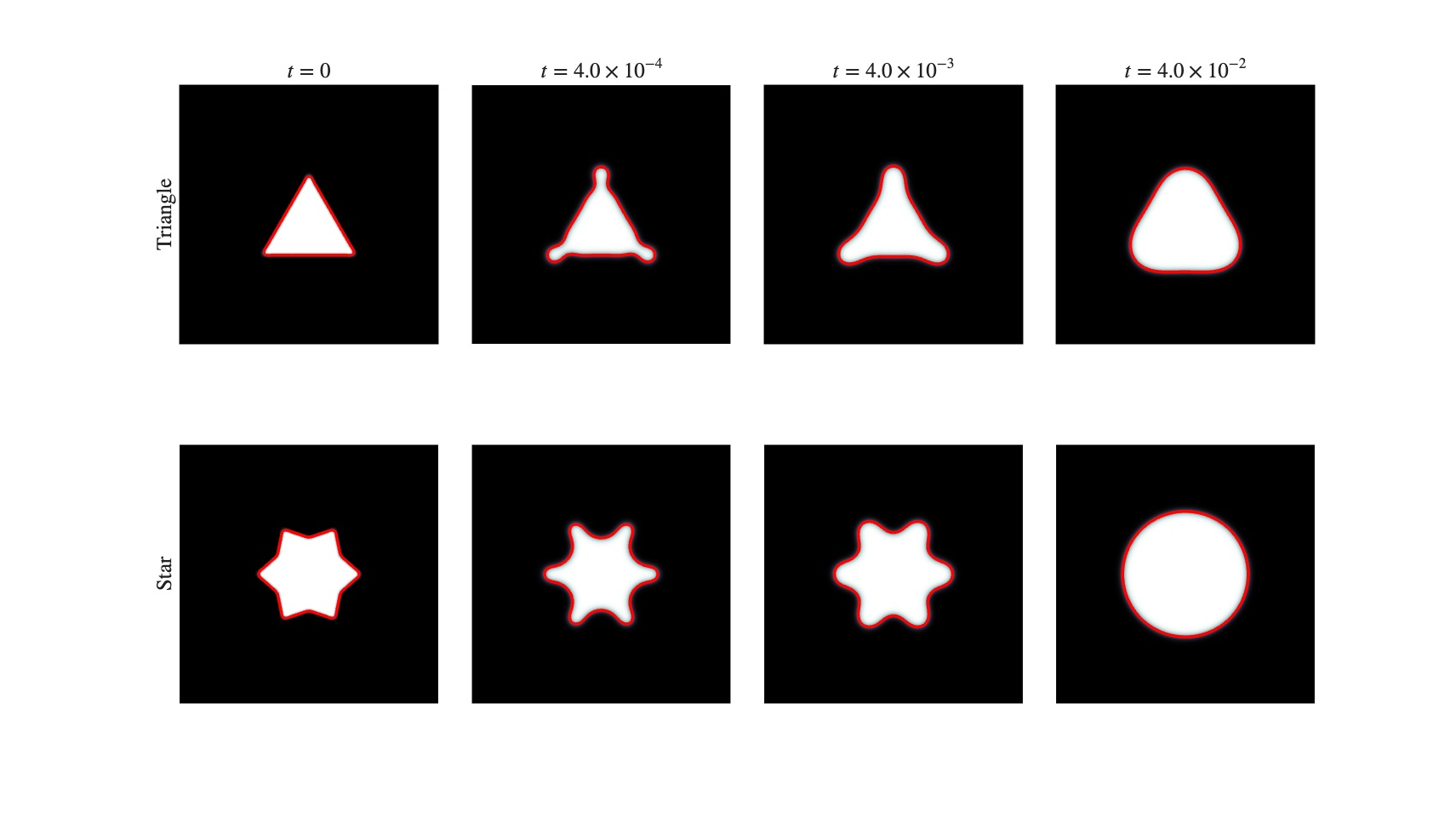}
    \caption{Single-vesicle growth under osmotic pressure ($\varepsilon = 0.01$, $N=2048$). Both morphologies exhibit strong rounding and evolution towards circular equilibrium configurations.}
    \label{fig:single_vesicle_gr_robustness}
\end{figure}

\noindent\textit{Shrinkage dynamics.} 
Shrinkage under hypotonic conditions \((\gamma_{\mathrm{bend}},\theta) = (0.1,30)\) presents a more stringent test. Figure~\ref{fig:single_vesicle_sh_robustness} demonstrates evolution from four geometrically diverse initial shapes: triangle, incomplete hexagon, star, and crescent configurations, spanning a range of topological features, including sharp corners, concave regions, and varying aspect ratios.

\begin{figure}[H]
    \centering
    \includegraphics[width=1.0\linewidth]{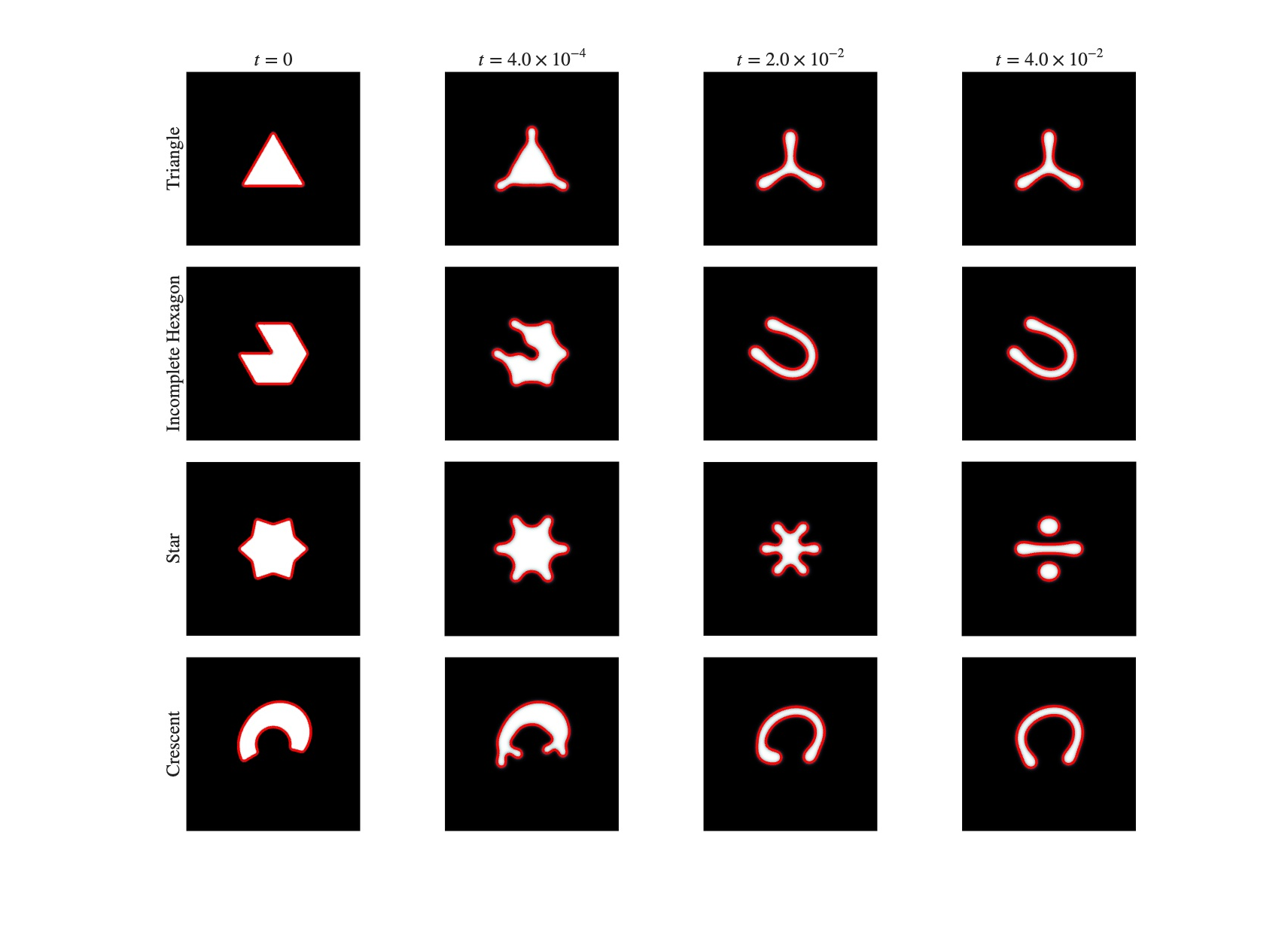}
    \includegraphics[width=0.85\linewidth]{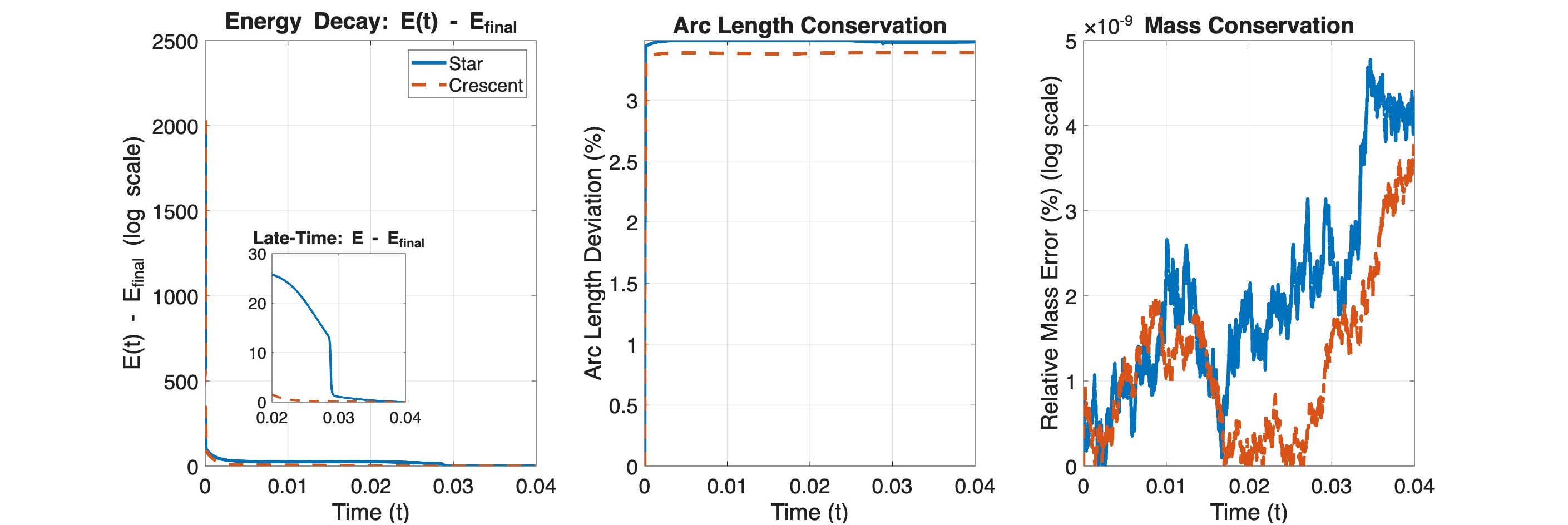}
    \caption{Single-vesicle shrinkage under osmotic pressure ($\varepsilon = 0.01$, $N=2048$). Top: Four morphologies demonstrating solver robustness. Bottom: physical energy evolution showing steady dissipation, arc length conservation, and mass conservation for star and crescent configurations.}
    \label{fig:single_vesicle_sh_robustness}
\end{figure}

Detailed diagnostics for two representative cases (star and crescent) are shown in the lower panels of Figure~\ref{fig:single_vesicle_sh_robustness}. Physical energy decays steadily toward steady states, with late-time convergence confirmed in the inset. Arc length deviations remain below 3.5\% throughout evolution. Relative mass conservation errors remain within $5 \times 10^{-11}$, confirming excellent numerical conservation properties. All four morphologies exhibit similar conservation characteristics.

\section{Conclusion and Future Work}
\label{sec:conclusion_and_future_work}

We have developed an efficient, unconditionally energy-stable numerical scheme for simulating vesicle deformation under osmotic pressure using a phase-field approach. The method addresses a fundamental computational bottleneck in variable-mobility Cahn--Hilliard equations arising in vesicle models: achieving constant-coefficient linear systems without sacrificing unconditional stability.

Our key innovation is a reformulated stabilization strategy that incorporates the stabilization term directly into the evolution equation rather than the chemical potential. This yields fully decoupled constant-coefficient elliptic operators solvable via fast DCT in $O(N^2 \log N)$ operations. The resulting method eliminates the iterative solvers required by previous approaches, achieving 6--15$\times$ speedup on production grids ($N \geq 2048$) while maintaining second-order accuracy. Unconditional energy stability is rigorously established in~\ref{sec:appendix_energy_stability_proof}, where we prove that the modified energy functional decays monotonically without time step restrictions. 

Comprehensive numerical experiments confirm the theoretical properties. The scheme exhibits second-order convergence in both space and time under successive refinement. Mass conservation is maintained to near-machine precision, with relative errors below $5 \times 10^{-11}$. Benchmark comparisons demonstrate close agreement with established nonlinear multigrid solvers for canonical growth and shrinkage scenarios. The modified energy decreases monotonically, as predicted by theory, while the physical energy exhibits steady dissipation in all test cases. The method handles complex geometries and non-convex initial shapes with topological changes robustly. Direct comparison between classical and constant-coefficient MSAV formulations confirms that efficiency gains come at no cost to accuracy.

Several directions for future work remain open. Coupling with Stokes equations would incorporate hydrodynamic effects in vesicle dynamics. Nonlinear osmotic models could capture asymmetric responses observed in red blood cell pathologies such as sickle cell disease. The constant-coefficient structure is well-suited for GPU acceleration and scalable parallel implementation, which would enable large-scale three-dimensional simulations. The combination of proven unconditional energy stability, second-order accuracy, optimal complexity scaling, and geometric flexibility makes the method particularly attractive for large-scale computational studies of biological membrane dynamics under osmotic stress.

\section*{Acknowledgments}
SL acknowledges partial support from the National Science Foundation (NSF), Division of Mathematical Sciences, Grant DMS-2309798. SMW acknowledges the support from NSF, Division of Mathematical Sciences, Grant DMS-2309547. JL acknowledges partial support from NSF, Division of Mathematical Sciences, Grant DMS-2309800.

\bibliographystyle{plain}
\bibliography{mybib}

@article{tang2023phase,
  title={Phase field modeling and computation of vesicle growth or shrinkage},
  author={Tang, Xiaoxia and Li, Shuwang and Lowengrub, John S and Wise, Steven M},
  journal={Journal of mathematical biology},
  volume={86},
  number={6},
  pages={97},
  year={2023},
  publisher={Springer}
}

@article{cheng2018multiple,
  title={Multiple scalar auxiliary variable (MSAV) approach and its application to the phase-field vesicle membrane model},
  author={Cheng, Qing and Shen, Jie},
  journal={SIAM Journal on Scientific Computing},
  volume={40},
  number={6},
  pages={A3982--A4006},
  year={2018},
  publisher={SIAM}
}

@article{wise2010unconditionally,
  title={Unconditionally stable finite difference, nonlinear multigrid simulation of the Cahn-Hilliard-Hele-Shaw system of equations},
  author={Wise, Steven M},
  journal={Journal of Scientific Computing},
  volume={44},
  number={1},
  pages={38--68},
  year={2010},
  publisher={Springer}
}

@article{shen2018scalar,
  title={The scalar auxiliary variable (SAV) approach for gradient flows},
  author={Shen, Jie and Xu, Jie and Yang, Jiang},
  journal={Journal of Computational Physics},
  volume={353},
  pages={407--416},
  year={2018},
  publisher={Elsevier}
}

@article{cahn1958free,
  title={Free energy of a nonuniform system. I. Interfacial free energy},
  author={Cahn, John W and Hilliard, John E},
  journal={The Journal of chemical physics},
  volume={28},
  number={2},
  pages={258--267},
  year={1958},
  publisher={American Institute of Physics}
}

@article{du2005phase,
  title={A phase field formulation of the Willmore problem},
  author={Du, Qiang and Liu, Chun and Ryham, Rolf and Wang, Xiaoqiang},
  journal={Nonlinearity},
  volume={18},
  number={3},
  pages={1249},
  year={2005},
  publisher={IOP Publishing}
}

@article{giga2017variational,
  title={Variational modeling and complex fluids},
  author={Giga, Mi-Ho and Kirshtein, Arkadz and Liu, Chun},
  journal={Handbook of mathematical analysis in mechanics of viscous fluids},
  pages={1--41},
  year={2017},
  publisher={Springer International Publishing, Cham}
}

@article{lowengrub2009phase,
  title={Phase-field modeling of the dynamics of multicomponent vesicles: Spinodal decomposition, coarsening, budding, and fission},
  author={Lowengrub, John S and R{\"a}tz, Andreas and Voigt, Axel},
  journal={Physical Review E—Statistical, Nonlinear, and Soft Matter Physics},
  volume={79},
  number={3},
  pages={031926},
  year={2009},
  publisher={APS}
}

@book{leveque2007finite,
  title={Finite difference methods for ordinary and partial differential equations: steady-state and time-dependent problems},
  author={LeVeque, Randall J},
  year={2007},
  publisher={SIAM}
}

@article{chen2019fast,
  title={Fast, provably unconditionally energy stable, and second-order accurate algorithms for the anisotropic Cahn--Hilliard model},
  author={Chen, Chuanjun and Yang, Xiaofeng},
  journal={Computer Methods in Applied Mechanics and Engineering},
  volume={351},
  pages={35--59},
  year={2019},
  publisher={Elsevier}
}

@article{bretin2023mobility,
  title={A mobility-SAV approach for a Cahn-Hilliard equation with degenerate mobilities},
  author={Bretin, Elie and Calatroni, Luca and Masnou, Simon},
  journal={arXiv preprint arXiv:2306.15329},
  year={2023}
}

@article{guo2025phase,
  title={Phase-Field Modeling and Energy-Stable Schemes for Osmotic Flow through Semi-Permeable},
  author={Guo, Ruihan and Shen, Jie and Xu, Shixin and Xu, Xianmin},
  journal={arXiv preprint arXiv:2506.11374},
  year={2025}
}

@article{zhang2024existence,
  author  = {ZHANG Senxiang and FENG Mei and LUO Hong},
  title   = {Existence and regularity of the solution of the triplet phase separation equation},
  journal = {Journal of Hainan University (Natural Science)},
  volume  = {42},
  number  = {4},
  pages   = {353--362},
  year    = {2024},
  doi     = {10.15886/j.cnki.hdxbzkb.2024012801}
}

@article{elani2014vesicle,
  title={Vesicle-based artificial cells as chemical microreactors with spatially segregated reaction pathways},
  author={Elani, Yuval and Law, Robert V and Ces, Oscar},
  journal={Nature communications},
  volume={5},
  number={1},
  pages={5305},
  year={2014},
  publisher={Nature Publishing Group UK London}
}

@article{mori2011model,
  title={A model of electrodiffusion and osmotic water flow and its energetic structure},
  author={Mori, Yoichiro and Liu, Chun and Eisenberg, Robert S},
  journal={Physica D: Nonlinear Phenomena},
  volume={240},
  number={22},
  pages={1835--1852},
  year={2011},
  publisher={Elsevier}
}

@article{vogl2014effect,
  title={The effect of glass-forming sugars on vesicle morphology and water distribution during drying},
  author={Vogl, CJ and Miksis, MJ and Davis, SH and Salac, D},
  journal={Journal of The Royal Society Interface},
  volume={11},
  number={99},
  pages={20140646},
  year={2014},
  publisher={The Royal Society}
}

@article{quaife2021hydrodynamics,
  title={Hydrodynamics of a semipermeable inextensible membrane under flow and confinement},
  author={Quaife, Bryan and Gannon, Ashley and Young, Y-N},
  journal={Physical Review Fluids},
  volume={6},
  number={7},
  pages={073601},
  year={2021},
  publisher={APS}
}

@article{seifert1997configurations,
  title={Configurations of fluid membranes and vesicles},
  author={Seifert, Udo},
  journal={Advances in physics},
  volume={46},
  number={1},
  pages={13--137},
  year={1997},
  publisher={Taylor \& Francis}
}

@article{veerapaneni2009boundary,
  title={A boundary integral method for simulating the dynamics of inextensible vesicles suspended in a viscous fluid in 2D},
  author={Veerapaneni, Shravan K and Gueyffier, Denis and Zorin, Denis and Biros, George},
  journal={Journal of Computational Physics},
  volume={228},
  number={7},
  pages={2334--2353},
  year={2009},
  publisher={Elsevier}
}

@article{veerapaneni2009numerical,
  title={A numerical method for simulating the dynamics of 3D axisymmetric vesicles suspended in viscous flows},
  author={Veerapaneni, Shravan K and Gueyffier, Denis and Biros, George and Zorin, Denis},
  journal={Journal of Computational Physics},
  volume={228},
  number={19},
  pages={7233--7249},
  year={2009},
  publisher={Elsevier}
}

@article{sohn2010dynamics,
  title={Dynamics of multicomponent vesicles in a viscous fluid},
  author={Sohn, Jin Sun and Tseng, Yu-Hau and Li, Shuwang and Voigt, Axel and Lowengrub, John S},
  journal={Journal of Computational Physics},
  volume={229},
  number={1},
  pages={119--144},
  year={2010},
  publisher={Elsevier}
}

@article{salac2011level,
  title={A level set projection model of lipid vesicles in general flows},
  author={Salac, David and Miksis, Michael},
  journal={Journal of Computational Physics},
  volume={230},
  number={22},
  pages={8192--8215},
  year={2011},
  publisher={Elsevier}
}

@article{liu2017dynamics,
  title={Dynamics of a multicomponent vesicle in shear flow},
  author={Liu, Kai and Marple, Gary R and Allard, Jun and Li, Shuwang and Veerapaneni, Shravan and Lowengrub, John},
  journal={Soft matter},
  volume={13},
  number={19},
  pages={3521--3531},
  year={2017},
  publisher={Royal Society of Chemistry}
}

@article{liu2014nonlinear,
  title={Nonlinear simulations of vesicle wrinkling},
  author={Liu, Kai and Li, Shuwang},
  journal={Mathematical Methods in the Applied Sciences},
  volume={37},
  number={8},
  pages={1093--1112},
  year={2014},
  publisher={Wiley Online Library}
}

@article{du2004phase,
  author = {Du, Qiang and Liu, Chun and Wang, Xiaoqiang},
  title = {A phase field approach in the numerical study of the elastic bending energy for vesicle membranes},
  journal = {Journal of Computational Physics},
  volume = {198},
  number = {2},
  pages = {450--468},
  year = {2004},
  doi = {10.1016/j.jcp.2004.01.029}
}

@article{wang2008modelling,
  title={Modelling and simulations of multi-component lipid membranes and open membranes via diffuse interface approaches},
  author={Wang, Xiaoqiang and Du, Qiang},
  journal={Journal of mathematical biology},
  volume={56},
  number={3},
  pages={347--371},
  year={2008},
  publisher={Springer}
}

@article{yang2017numerical,
  title={Numerical approximations for a three-component Cahn--Hilliard phase-field model based on the invariant energy quadratization method},
  author={Yang, Xiaofeng and Zhao, Jia and Wang, Qi and Shen, Jie},
  journal={Mathematical Models and Methods in Applied Sciences},
  volume={27},
  number={11},
  pages={1993--2030},
  year={2017},
  publisher={World Scientific}
}

@article{huang2023structure,
  title={A structure-preserving, upwind-SAV scheme for the degenerate Cahn--Hilliard equation with applications to simulating surface diffusion},
  author={Huang, Qiong-Ao and Jiang, Wei and Yang, Jerry Zhijian and Yuan, Cheng},
  journal={Journal of Scientific Computing},
  volume={97},
  number={3},
  pages={64},
  year={2023},
  publisher={Springer}
}

@article{orizaga2024imex,
  title={Imex methods for thin-film equations and cahn--hilliard equations with variable mobility},
  author={Orizaga, Saulo and Witelski, Thomas},
  journal={Computational Materials Science},
  volume={243},
  pages={113145},
  year={2024},
  publisher={Elsevier}
}

@article{xiao2023three,
  title={Three-dimensional numerical study on wrinkling of vesicles in elongation flow based on the immersed boundary method},
  author={Xiao, Wang and Liu, Kai and Lowengrub, John and Li, Shuwang and Zhao, Meng},
  journal={Physical Review E},
  volume={107},
  number={3},
  pages={035103},
  year={2023},
  publisher={APS}
}

@article{hausser2013thermodynamically,
  title={Thermodynamically consistent models for two-component vesicles},
  author={Hau{\ss}er, Frank and Li, S and Lowengrub, J and Marth, W and R{\"a}tz, Andreas and Voigt, Axel},
  journal={Int. J. Biomath. Biostat},
  volume={2},
  number={1},
  pages={19--48},
  year={2013}
}

@article{strang1999discrete,
  title={The discrete cosine transform},
  author={Strang, Gilbert},
  journal={SIAM review},
  volume={41},
  number={1},
  pages={135--147},
  year={1999},
  publisher={SIAM}
}

@book{trefethen2000spectral,
  title={Spectral methods in MATLAB},
  author={Trefethen, Lloyd N},
  year={2000},
  publisher={SIAM}
}

@article{schumann1988fast,
  title={Fast Fourier transforms for direct solution of Poisson's equation with staggered boundary conditions},
  author={Schumann, Ulrich and Sweet, Roland A},
  journal={Journal of Computational Physics},
  volume={75},
  number={1},
  pages={123--137},
  year={1988},
  publisher={Elsevier}
}

@article{shen1994efficient,
  title={Efficient spectral-Galerkin method I. Direct solvers of second-and fourth-order equations using Legendre polynomials},
  author={Shen, Jie},
  journal={SIAM Journal on Scientific Computing},
  volume={15},
  number={6},
  pages={1489--1505},
  year={1994},
  publisher={SIAM}
}

@article{kwan2007efficient,
  title={An efficient direct parallel spectral-element solver for separable elliptic problems},
  author={Kwan, Yuen-Yick and Shen, Jie},
  journal={Journal of Computational Physics},
  volume={225},
  number={2},
  pages={1721--1735},
  year={2007},
  publisher={Elsevier}
}

@article{quan2022decreasing,
  title={A decreasing upper bound of energy for time-fractional phase-field equations},
  author={Quan, Chaoyu and Tang, Tao and Wang, Boyi and Yang, Jiang},
  journal={arXiv preprint arXiv:2202.12192},
  year={2022}
}

\appendix
\section{Complete Proof of Unconditional Energy Stability}
\label{sec:appendix_energy_stability_proof}

This appendix provides the proof of Theorem~\ref{thm:energy_stability}. The overall proof methodology follows the standard MSAV-BDF2 framework established by Cheng and Shen~\cite{cheng2018multiple}. We focus on the key novelty: the \textit{Cahn--Hilliard stabilization} term $\lambda(\Delta\psi^{n+1}-\Delta\psi^{*,n+1})$ producing the remainder, and we absorb it as a difference between two accumulated sums. The coupled Allen--Cahn--Cahn--Hilliard structure with cross-phase coupling $K(\phi,\psi)$ and $P(\phi,\psi)$ is handled naturally within the MSAV framework through the auxiliary variable $Z^{n+1}$.

\subsection{Notation and BDF2 Identities}

Throughout, we use standard $L^2$ inner products $(\cdot,\cdot)$ and norms $\|\cdot\|$. Extrapolated quantities at time level $n+1$ are denoted with superscript $*$, e.g., $M_{\psi}^* = M_{\psi}(\phi^{*,n+1}) > 0$, $K^* = K(\phi^{*,n+1},\psi^{*,n+1})$, $P^* = P(\phi^{*,n+1},\psi^{*,n+1})$. The proof relies on the following BDF2 identities~\cite{cheng2018multiple}:

\begin{lemma}[BDF2 Identities]
\label{lem:bdf2_identities}
For any scalar or vector sequence $\{a^n\}$, the following identities hold:
\begin{align}
2\bigl(3a^{n+1}-4a^n+a^{n-1},\, a^{n+1}\bigr)
&= \|a^{n+1}\|^2 - \|a^n\|^2 \notag\\
&\quad + \|2a^{n+1}-a^n\|^2
      - \|2a^n-a^{n-1}\|^2 \notag\\
&\quad + \|a^{n+1}-2a^n+a^{n-1}\|^2,
\label{eq:bdf2_id1} \\[0.5ex]
2\bigl(3a^{n+1}-4a^n+a^{n-1},\, 2a^n-a^{n-1}\bigr)
&=
\left(
\begin{aligned}
&\|a^{n+1}\|^2 + \|2a^{n+1}-a^n\|^2 \\
&\quad - 2\|a^{n+1}-a^n\|^2
\end{aligned}
\right) \notag\\
&\quad -
\left(
\begin{aligned}
&\|a^{n}\|^2 + \|2a^{n}-a^{n-1}\|^2 \\
&\quad - 2\|a^{n}-a^{n-1}\|^2
\end{aligned}
\right) \notag\\
&\quad - 3\|a^{n+1}-2a^n+a^{n-1}\|^2.
\label{eq:bdf2_id2}
\end{align}
\end{lemma}

\subsection{Weak Formulation and Completion of Squares}

Taking the $L^2$ inner product of~\eqref{eq:AC_BDF2} with $2\Delta t\,\mu^{n+1}$ and~\eqref{eq:CH_BDF2} with $2\Delta t\,\nu^{n+1}$, and using integration by parts with homogeneous Neumann boundary conditions:
\begin{multline}
(3\phi^{n+1}-4\phi^n+\phi^{n-1},\mu^{n+1})
+(3\psi^{n+1}-4\psi^n+\psi^{n-1},\nu^{n+1}) \\
=-2\Delta t \Big[M_{\phi}\|\mu^{n+1}\|^2
+(\nabla\nu^{n+1},M_{\psi}^*\nabla\nu^{n+1})
+\lambda(\nabla\nu^{n+1},\nabla\psi^{n+1}-\nabla\psi^{*,n+1})\Big].
\label{eq:weak_form_full}
\end{multline}

\paragraph{Handling the stabilization term} 
Define the auxiliary quantities:
\begin{align}
c^{n+1} &= \sqrt{M_{\psi}^*}\,\nabla\nu^{n+1}, \label{eq:c_def}\\
d^{n+1} &= \frac{\lambda(\nabla\psi^{n+1}-\nabla\psi^{*,n+1})}{\sqrt{M_{\psi}^*}}. \label{eq:d_def}
\end{align}
Completing the square, the Cahn--Hilliard dissipation terms in~\eqref{eq:weak_form_full} become:
\begin{align}
&(\nabla\nu^{n+1},M_{\psi}^*\nabla\nu^{n+1})
+\lambda(\nabla\nu^{n+1},\nabla\psi^{n+1}-\nabla\psi^{*,n+1}) \notag\\
&\qquad= \|c^{n+1}\|^2 + (c^{n+1},d^{n+1}) 
= \left\|c^{n+1}+\tfrac{1}{2}d^{n+1}\right\|^2 - \tfrac{1}{4}\|d^{n+1}\|^2.
\label{eq:completion_key}
\end{align}
Thus~\eqref{eq:weak_form_full} becomes:
\begin{multline}
(3\phi^{n+1}-4\phi^n+\phi^{n-1},\mu^{n+1})
+(3\psi^{n+1}-4\psi^n+\psi^{n-1},\nu^{n+1}) \\
=-2\Delta t \Big[M_{\phi}\|\mu^{n+1}\|^2
+\left\|c^{n+1}+\tfrac{1}{2}d^{n+1}\right\|^2\Big]
+\frac{\Delta t}{2}\|d^{n+1}\|^2.
\label{eq:weak_final}
\end{multline}

\subsection{Energy Decomposition}

We expand the left-hand side of~\eqref{eq:weak_final} by substituting the chemical potential definitions~\eqref{eq:mu_BDF2}--\eqref{eq:nu_BDF2}. 

\subsubsection{Standard Energy Terms}

For all standard (implicitly treated) energy terms (surface, bending, area, and scalar auxiliary variables $V, U, W$), we apply integration by parts, followed by identities~\eqref{eq:bdf2_id1}--\eqref{eq:bdf2_id2}. Each contribution decomposes into:
\begin{itemize}
\item Energy differences between consecutive time levels: $\|a^{n+1}\|^2 - \|a^n\|^2 + \|2a^{n+1}-a^n\|^2 - \|2a^n-a^{n-1}\|^2$;
\item Non-negative numerical dissipation: $\|a^{n+1}-2a^n+a^{n-1}\|^2 \geq 0$,
\end{itemize}
where $a$ represents $\Delta\phi$, $\nabla\phi$, $\phi$, or scalar auxiliary variables. The algebraic manipulations are identical to~\cite[Theorem 3.2]{cheng2018multiple}.

\subsubsection{Cross-Phase Coupling Through Auxiliary Variable $Z$}

From the Allen--Cahn chemical potential~\eqref{eq:mu_BDF2}, the coupling term $K^*$ contributes:
\begin{equation}
(3\phi^{n+1}-4\phi^n+\phi^{n-1}, Z^{n+1}K^*).
\label{eq:ac_coupling}
\end{equation}

From the Cahn--Hilliard side, the coupling term $P^*$ contributes:
\begin{equation}
(3\psi^{n+1}-4\psi^n+\psi^{n-1}, Z^{n+1}P^*).
\label{eq:ch_coupling}
\end{equation}

Combining~\eqref{eq:ac_coupling}--\eqref{eq:ch_coupling} with the auxiliary variable evolution equation~\eqref{eq:Z_BDF2}, we obtain:
\begin{equation}
\begin{aligned}
&(3\phi^{n+1}-4\phi^n+\phi^{n-1}, Z^{n+1}K^*)+(3\psi^{n+1}-4\psi^n+\psi^{n-1}, Z^{n+1}P^*)\\
&=(3Z^{n+1}-4Z^{n}+Z^{n-1},Z^{n+1}).
\end{aligned}
\label{eq:coupling_combined}
\end{equation}
Applying identity~\eqref{eq:bdf2_id1} with $a=Z$ to expand the right-hand side produces energy differences and numerical dissipation associated with the auxiliary variable $Z$.

\subsubsection{Treatment of the Stabilization Remainder}

The key novelty in our energy analysis is handling the positive remainder $\frac{\Delta t}{2}\|d^{n+1}\|^2$ that appears on the right-hand side of equation~\eqref{eq:weak_final}. Following the spirit of modified energy approaches for fractional equations \cite{quan2022decreasing}, we rewrite this single-step contribution as a difference between two cumulative sums:
\begin{equation}
\frac{\Delta t}{2}\|d^{n+1}\|^2 = \left(\sum_{j=2}^{n+1} \tfrac{\Delta t}{2}\,\|d^{j}\|^2\right) - \left(\sum_{j=2}^{n} \tfrac{\Delta t}{2}\,\|d^{j}\|^2\right).
\label{eq:telescoping_stab}
\end{equation}

Moving this telescoping difference to the left-hand side of~\eqref{eq:weak_final} and incorporating it into the energy balance yields:
\begin{multline}
(3\phi^{n+1}-4\phi^n+\phi^{n-1},\mu^{n+1})
+(3\psi^{n+1}-4\psi^n+\psi^{n-1},\nu^{n+1}) \\
- \left(\sum_{j=2}^{n+1} \tfrac{\Delta t}{2}\,\|d^{j}\|^2\right) + \left(\sum_{j=2}^{n} \tfrac{\Delta t}{2}\,\|d^{j}\|^2\right) \\
=-2\Delta t \Big[M_{\phi}\|\mu^{n+1}\|^2
+\left\|c^{n+1}+\tfrac{1}{2}d^{n+1}\right\|^2\Big].
\label{eq:weak_with_stab_sum}
\end{multline}

The cumulative sum $-\sum_{j=2}^{n+1} \tfrac{\Delta t}{2}\,\|d^{j}\|^2$ now appears as part of the modified energy at time level $n+1$, while $-\sum_{j=2}^{n} \tfrac{\Delta t}{2}\,\|d^{j}\|^2$ appears in the energy at time level $n$. This structure ensures that the stabilization contribution accumulates properly across time steps while maintaining the overall energy dissipation property.

\subsection{Assembly and Final Energy Balance}

Collecting all contributions from the standard terms, cross-phase coupling, and stabilization, the left-hand side of~\eqref{eq:weak_with_stab_sum} equals:
\begin{equation}
E^{n+1,n}_{\mathrm{mod}} - E^{n,n-1}_{\mathrm{mod}} + K^n,
\label{eq:energy_difference}
\end{equation}
where $E^{n+1,n}_{\mathrm{mod}}$ is the modified discrete energy~\eqref{eq:modified_discrete_energy}, and $K^n$ collects all non-negative numerical dissipation arising from the BDF2 second-order differences. Explicitly:
\begin{align}
K^n &= \frac{1}{2}\gamma_{\mathrm{bend}} \frac{3\sqrt{2} \varepsilon}{8} \|\Delta\phi^{n+1}-2\Delta\phi^n+\Delta\phi^{n-1}\|^2 \notag\\
&\quad+ \frac{1}{2}\left(\gamma_{\mathrm{surf}} \frac{3\sqrt{2} \varepsilon}{4} + \theta\right) \|\nabla\phi^{n+1}-2\nabla\phi^n+\nabla\phi^{n-1}\|^2 \notag\\
&\quad+ \frac{3\tilde{\theta}}{2} \|\nabla\phi^{n+1}-2\nabla\phi^n+\nabla\phi^{n-1}\|^2 \notag\\
&\quad+ \frac{1}{2}\gamma_{\mathrm{surf}} \frac{3\sqrt{2} \beta}{4\varepsilon} \|\phi^{n+1}-2\phi^n+\phi^{n-1}\|^2 \notag\\
&\quad+ \frac{1}{2}\gamma_{\mathrm{surf}} \frac{3\sqrt{2}}{2} (V^{n+1}-2V^n+V^{n-1})^2 \notag\\
&\quad+ \frac{1}{2}\gamma_{\mathrm{area}} (U^{n+1}-2U^n+U^{n-1})^2 \notag\\
&\quad+ \frac{1}{2}\gamma_{\mathrm{bend}} \frac{3\sqrt{2}}{8\varepsilon} (W^{n+1}-2W^n+W^{n-1})^2 \notag\\
&\quad+ \frac{1}{2}(Z^{n+1}-2Z^n+Z^{n-1})^2 \geq 0.
\label{eq:K_n_explicit}
\end{align}

Combining~\eqref{eq:energy_difference} with~\eqref{eq:weak_with_stab_sum} and dividing by $2\Delta t$:
\begin{equation}
\frac{E^{n+1,n}_{\mathrm{mod}} - E^{n,n-1}_{\mathrm{mod}}}{2\Delta t} + \frac{K^n}{2\Delta t}
= -\Big[M_{\phi}\|\mu^{n+1}\|^2 + \big\|c^{n+1}+\tfrac{1}{2}d^{n+1}\big\|^2\Big] \leq 0.
\label{eq:energy_balance_final}
\end{equation}
This establishes the unconditional energy stability stated in Theorem~\ref{thm:energy_stability}.
\qed
\section{Numerical Solution of the Fully Discretized System}
\label{sec:appendix_solution_procedure}

In this section, we detail the procedure for solving the coupled linear systems arising from the fully discretized schemes. The core of our approach is to efficiently reduce the large system for the field variables ($\phi^{n+1}$, $\psi^{n+1}$) to a small, dense linear system for a few unknown scalar integrals. This strategy, common in SAV-based methods, avoids the need to solve a large, non-local system directly.

The procedure is nearly identical for both the second-order BDF2 and first-order BDF1 schemes. We present the derivation for the BDF2 case; the BDF1 case follows by substituting the BDF1 operators for their BDF2 counterparts.

\subsection*{Step 1: Expressing SAVs in Terms of Field Variables}
First, we rearrange the discrete evolution equations for the SAVs (Equations~\eqref{eq:V_BDF2}--\eqref{eq:Z_BDF2}) to express the new SAVs ($V^{n+1}$, $U^{n+1}$, $W^{n+1}$, $Z^{n+1}$) as functions of known past values and unknown integrals involving the new field variables $\phi^{n+1}$ and $\psi^{n+1}$. For example, for $V^{n+1}$:
\begin{equation}
    V^{n+1} = \frac{1}{3}\left(4V^n - V^{n-1} + \int_{\Omega} S(\phi^{*,n+1})(3\phi^{n+1} - 4\phi^n + \phi^{n-1})\,d\mathbf{x}\right).
\end{equation}

\subsection*{Step 2: Rearrangement into Linear System Form}

With the SAVs expressed in terms of the field variables, we substitute them back into the discretized Allen--Cahn and Cahn--Hilliard equations. This substitution results in a large, fully coupled system. The crucial task is to algebraically rearrange this system to isolate all terms involving the unknown fields at the new time step, $\phi^{n+1}$ and $\psi^{n+1}$.

Our strategy places all implicit terms on the left-hand side (LHS) and all explicit terms (i.e., terms from previous time steps $n$ and $n-1$) on the right-hand side (RHS). This process transforms the equations into a structure amenable to linear solution, which we detail in the subsequent steps. For clarity, we present the rearrangement for each governing equation separately.

\vspace{1em}

\noindent\textbf{Allen--Cahn Equation}\\
After substitution and rearrangement, the Allen--Cahn equation takes the following form, where the LHS contains all operators acting on the unknown fields $\phi^{n+1}$ and $\psi^{n+1}$:
\begin{align}
\Biggl\{
&\frac{3}{2\Delta t}
 + M_{\phi}\gamma_{\text{surf}}\frac{3\sqrt{2}\beta}{4\varepsilon}
\notag\\
&\quad
 - M_{\phi}\Bigl(\gamma_{\text{surf}}\frac{3\sqrt{2}\varepsilon}{4}+\theta\Bigr)\Delta
 + M_{\phi}\gamma_{\text{bend}}\frac{3\sqrt{2}\varepsilon}{8}\Delta^2
\Biggr\}\phi^{n+1}
\notag\\
&+ M_{\phi}\gamma_{\text{surf}}\frac{3\sqrt{2}}{2}
   \Bigl(\int_{\Omega} S(\phi^{*,n+1})\phi^{n+1}\,d\mathbf{x}\Bigr)
   S(\phi^{*,n+1})
\notag\\
&+ M_{\phi}\gamma_{\text{area}}
   \Bigl(\int_{\Omega} H(\phi^{*,n+1})\phi^{n+1}\,d\mathbf{x}\Bigr)
   H(\phi^{*,n+1})
\notag\\
&+ M_{\phi}\gamma_{\text{bend}}\frac{3\sqrt{2}}{8\varepsilon}
   \Bigl(\int_{\Omega} Q(\phi^{*,n+1})\phi^{n+1}\,d\mathbf{x}\Bigr)
   Q(\phi^{*,n+1})
\notag\\
&+ M_{\phi}
   \Bigl(\int_{\Omega} K(\phi^{*,n+1},\psi^{*,n+1})\phi^{n+1}\,d\mathbf{x}\Bigr)
   K(\phi^{*,n+1},\psi^{*,n+1})
\notag\\
&+ M_{\phi}
   \Bigl(\int_{\Omega} P(\phi^{*,n+1},\psi^{*,n+1})\psi^{n+1}\,d\mathbf{x}\Bigr)
   K(\phi^{*,n+1},\psi^{*,n+1})
= f^n .
\end{align}
The RHS, denoted by $f^n$, consolidates all explicit terms:
\begin{align}
f^n
&=\frac{4\phi^n-\phi^{n-1}}{2\Delta t}
 - M_{\phi}\Biggl\{
\left(\theta+\gamma_{\text{bend}}\frac{3\sqrt{2}}{4\varepsilon}\right)\Delta\phi^{*,n+1}
\notag\\
&\quad
+\gamma_{\text{surf}} \frac{3 \sqrt{2}}{6}
\Bigl[4V^{n}-V^{n-1}
+\int_{\Omega} S\!\left(\phi^{*,n+1}\right)\left(-4\phi^n+\phi^{n-1}\right)\,d\mathbf{x}\Bigr]
\, S(\phi^{*,n+1})
\notag\\
&\quad
+\frac{\gamma_{\text{area}}}{3}
\Bigl[4U^{n}-U^{n-1}
+\int_{\Omega} H\!\left(\phi^{*,n+1}\right)\left(-4\phi^n+\phi^{n-1}\right)\,d\mathbf{x}\Bigr]
\, H(\phi^{*,n+1})
\notag\\
&\quad
+\gamma_{\text{bend}} \frac{3 \sqrt{2}}{24 \varepsilon}
\Bigl[4W^{n}-W^{n-1}
+\int_{\Omega} Q\!\left(\phi^{*,n+1}\right)\left(-4\phi^n+\phi^{n-1}\right)\,d\mathbf{x}\Bigr]
\notag\\
&\quad\quad \times Q(\phi^{*,n+1})
\notag\\
&\quad
+ \frac{1}{3}\Biggl[4Z^{n}-Z^{n-1}
+\int_{\Omega} K\!\left(\phi^{*,n+1},\psi^{*,n+1}\right)\left(-4\phi^n+\phi^{n-1}\right)\,d\mathbf{x}
\notag\\
&\quad\quad
+\int_{\Omega} P\!\left(\phi^{*,n+1},\psi^{*,n+1}\right)\left(-4\psi^n+\psi^{n-1}\right)\,d\mathbf{x}
\Biggr]K(\phi^{*,n+1},\psi^{*,n+1})
\Biggr\}.
\end{align}

\vspace{1em}

\noindent\textbf{Cahn--Hilliard Equation}\\
A similar procedure is applied to the Cahn--Hilliard equation. To maintain linearity, the phase-dependent mobility $M_{\psi}(\phi^{n+1})$ is treated semi-implicitly by evaluating it at $\phi^{*,n+1}$. The final rearranged form is
\begin{align}
    &\left(\frac{3}{2\Delta t}
    -\lambda\Delta\right)\psi^{n+1}\notag\\
    &-\left[\int_{\Omega}K(\phi^{*,n+1},\psi^{*,n+1})\phi^{n+1}\,d\mathbf{x}\right]\nabla \cdot \left( M_{\psi}(\phi^{*,n+1}) \nabla P\left(\phi^{*,n+1},\psi^{*,n+1}\right)    \right)\notag\\
    &-\left[\int_{\Omega}P(\phi^{*,n+1},\psi^{*,n+1})\psi^{n+1}\,d\mathbf{x}\right]\nabla \cdot \left( M_{\psi}(\phi^{*,n+1}) \nabla P\left(\phi^{*,n+1},\psi^{*,n+1}\right)    \right)\notag\\
    &=g^n.
\end{align}
The RHS, $g^n$, containing all explicit terms, is given by
\begin{align}
    g^n &= \frac{4\psi^n-\psi^{n-1}}{2\Delta t}-\lambda\Delta\psi^{*,n+1}+\nabla \cdot \left( M_{\psi}(\phi^{*,n+1}) \nabla \tilde{\nu}^{n+1} \right),
\end{align}
where $\tilde{\nu}^{n+1}$, the explicit portion of the chemical potential, is defined as
\begin{align}
    \tilde{\nu}^{n+1}&= \frac{1}{3}\Biggl[4Z^{n}-Z^{n-1}+\int_{\Omega} K\left(\phi^{*,n+1},\psi^{*,n+1}\right)\left(-4\phi^n+\phi^{n-1}\right) \,d\mathbf{x}\notag\\
    &\quad+\int_{\Omega} P\left(\phi^{*,n+1},\psi^{*,n+1}\right)\left(-4\psi^n+\psi^{n-1}\right) \,d\mathbf{x}\Biggr]P\left(\phi^{*,n+1},\psi^{*,n+1}\right).       
\end{align}

This rearrangement successfully prepares the system for the decoupled solution procedure described next.

\subsection*{Step 3: Decoupling via Linear Operators}
To solve the coupled system, we first decouple the spatial differential operators from the integral terms. This is achieved by defining two linear, constant-coefficient elliptic operators, $\chi$ and $\zeta$, which collect all implicit linear terms for $\phi$ and $\psi$, respectively.

The operator for the Allen--Cahn equation, $\chi$, is defined as
\begin{align}
    \chi(\phi) &:= C_{1,\chi} \phi - C_{2,\chi} \Delta \phi + C_{3,\chi}\Delta^2 \phi,\label{eq:chi_operator}
\end{align}
where 
\begin{align}
C_{1,\chi}&=\frac{3}{2\Delta t} + M_{\phi} \gamma_{\text{surf}} \frac{3\sqrt{2}\beta}{4\varepsilon}, \\
C_{2,\chi}&=M_{\phi} \left[ \gamma_{\text{surf}} \frac{3\sqrt{2}\varepsilon}{4} + \theta \right], \\
C_{3,\chi}&=M_{\phi} \gamma_{\text{bend}} \frac{3\sqrt{2} \varepsilon}{8}.
\end{align}

The operator for the Cahn--Hilliard equation, $\zeta$, is defined as
\begin{equation}
    \zeta(\psi) := \left( \frac{3}{2\Delta t} - \lambda \Delta \right) \psi.
\end{equation}

\begin{remark}[Well-posedness of Linear Operators]
The discrete operators $\chi$ and $\zeta$ are well-posed under the homogeneous Neumann boundary conditions implemented through ghost cells in the cell-centered finite difference framework.

Given the expression $\chi(\phi) = C_{1,\chi} \phi - C_{2,\chi} \Delta_h \phi + C_{3,\chi}\Delta_h^2 \phi$, we analyze the properties of the discrete Laplacian operator, $\Delta_h$. For a system with Neumann boundary conditions, the eigenvalues $\lambda_{k,\ell}$ of $\Delta_h$ are given by:
$$\lambda_{k,\ell} = -\frac{4}{h^2}\left[\sin^2\left(\frac{k\pi h}{2L_x}\right) + \sin^2\left(\frac{\ell\pi h}{2L_y}\right)\right]$$
where the indices run from $k=0, \dots, m-1$ and $\ell=0, \dots, n-1$.

Note that these eigenvalues are non-positive (i.e., $-\lambda_{k,\ell} \ge 0$). The specific case $\lambda_{0,0} = 0$ corresponds to the constant mode. The spectrum of the discrete operator $\chi$ consists of eigenvalues:
$$\sigma_{k,\ell} = C_{1,\chi} - C_{2,\chi}\lambda_{k,\ell} + C_{3,\chi}\lambda_{k,\ell}^2.$$
For the constant mode ($k=\ell=0$), we have $\sigma_{0,0} = C_{1,\chi} > 0$. For all other modes, since $-\lambda_{k,\ell} > 0$ and all coefficients are positive, we obtain
$$\sigma_{k,\ell} = C_{1,\chi} + C_{2,\chi}|\lambda_{k,\ell}| + C_{3,\chi}\lambda_{k,\ell}^2 > C_{1,\chi} > 0,$$
ensuring invertibility of the discrete system.

For $\zeta(\psi) = \frac{3}{2\Delta t} \psi - \lambda \Delta_h \psi$:
This is a discrete Poisson-type operator. Using the same eigenvalues of $\Delta_h$, the spectrum of $\zeta$ is:
$$\tau_{k,\ell} = \frac{3}{2\Delta t} - \lambda\lambda_{k,\ell} = \frac{3}{2\Delta t} + \lambda|\lambda_{k,\ell}| \geq \frac{3}{2\Delta t} > 0,$$
guaranteeing unique solvability. The strictly positive spectrum ensures that the discrete linear systems can be efficiently solved using fast direct methods such as the discrete cosine transform (DCT) for problems with Neumann boundary conditions.
\end{remark}

Using these definitions, the fully discretized system can be formally rewritten. For brevity, we introduce the shorthand notation $S^* = S(\phi^{*,n+1})$, $H^* = H(\phi^{*,n+1})$, $Q^* = Q(\phi^{*,n+1})$, $K^* = K(\phi^{*,n+1},\psi^{*,n+1})$, $P^* = P(\phi^{*,n+1},\psi^{*,n+1})$, and the following scalar coefficients:
\begin{align}
    c_{gs} &= M_{\phi}\gamma_{\text{surf}}\frac{3\sqrt{2}}{2}, & c_{ga} &= M_{\phi}\gamma_{\text{area}}, \\
    c_{gb} &= M_{\phi}\gamma_{\text{bend}}\frac{3\sqrt{2}}{8\varepsilon}, & c_{go} &= M_{\phi}.
\end{align}
The coupled Allen--Cahn--Cahn--Hilliard system can then be written as
\begin{align}
    \phi^{n+1} & + c_{gs} \left[ \int_{\Omega} S^* \phi^{n+1} \,d\mathbf{x} \right] \chi^{-1}(S^*) \notag \\
    &  + c_{ga} \left[ \int_{\Omega} H^* \phi^{n+1} \,d\mathbf{x} \right] \chi^{-1}(H^*) \notag \\
    &  + c_{gb} \left[ \int_{\Omega} Q^* \phi^{n+1} \,d\mathbf{x} \right] \chi^{-1}(Q^*) \notag \\
    &  + c_{go} \left[ \int_{\Omega} K^* \phi^{n+1} \,d\mathbf{x} \right] \chi^{-1}(K^*)\notag\\
    & + c_{go} \left[ \int_{\Omega} P^* \psi^{n+1} \,d\mathbf{x} \right] \chi^{-1}(K^*)\notag\\
    &= \chi^{-1}(f^n),\label{eq:decoupled_AC}\\
    \psi^{n+1}
    &- \left[ \int_{\Omega} K^* \phi^{n+1} \,d\mathbf{x} \right] \zeta^{-1} \left( \nabla \cdot (M_{\psi}(\phi^{*,n+1}) \nabla P^*) \right)\notag\\    
    &-  \left[ \int_{\Omega} P^* \psi^{n+1} \,d\mathbf{x} \right] \zeta^{-1} \left( \nabla \cdot (M_{\psi}(\phi^{*,n+1}) \nabla P^*) \right)\notag\\
    &= \zeta^{-1}(g^n).\label{eq:decoupled_CH}
\end{align}

\subsection*{Step 4: Derivation of the Small Linear System}
The final step in the solution procedure is to solve for the unknown scalar integrals. This is achieved by reducing the large, coupled PDE system to a small, dense algebraic system.

First, we define the unknown integrals that appear in the expressions for $\phi^{n+1}$ and $\psi^{n+1}$:
\begin{align}
    \tilde{B} &= \int_{\Omega} S(\phi^{*,n+1}) \phi^{n+1} \, d\mathbf{x}, \quad
    \tilde{C} = \int_{\Omega} H(\phi^{*,n+1}) \phi^{n+1} \, d\mathbf{x}, \\
    \tilde{D} &= \int_{\Omega} Q(\phi^{*,n+1}) \phi^{n+1} \, d\mathbf{x}, \quad
    \tilde{E} = \int_{\Omega} K(\phi^{*,n+1}, \psi^{*,n+1}) \phi^{n+1} \, d\mathbf{x}, \\
    \tilde{G} &= \int_{\Omega} P(\phi^{*,n+1}, \psi^{*,n+1}) \psi^{n+1} \, d\mathbf{x}.
\end{align}

By taking the inner product of the formally solved Allen--Cahn and Cahn--Hilliard equations (Equations~\eqref{eq:decoupled_AC}--\eqref{eq:decoupled_CH}) with each of the relevant basis functions, we obtain a coupled $5 \times 5$ linear system for the unknowns $\tilde{B}, \tilde{C}, \tilde{D}, \tilde{E}, \tilde{G}$. This system is written in compact matrix form:
\begin{equation}
    \mathbf{M}_{\phi,\psi} \begin{pmatrix} \tilde{B} \\ \tilde{C} \\ \tilde{D} \\ \tilde{E} \\ \tilde{G}\end{pmatrix} = \mathbf{b}_{\phi,\psi}.
\end{equation}

To express the matrix and vector compactly, we use the inner product notation $(f,g) = \int_\Omega fg \,d\mathbf{x}$ and define $\hat{S}^* = \chi^{-1}(S^*)$, $\hat{H}^* = \chi^{-1}(H^*)$, $\hat{Q}^* = \chi^{-1}(Q^*)$, and $\hat{K}^* = \chi^{-1}(K^*)$. Additionally, we introduce $\Lambda_{P^*}=\zeta^{-1} \left( \nabla \cdot (M_{\psi}(\phi^{*,n+1}) \nabla P^*) \right)$.

The right-hand side vector $\mathbf{b}_{\phi,\psi}$ is given by
\begin{equation}
    \mathbf{b}_{\phi,\psi} = \begin{pmatrix}
        (S^*, \chi^{-1}(f^n)) \\ (H^*, \chi^{-1}(f^n)) \\ (Q^*, \chi^{-1}(f^n)) \\ (K^*, \chi^{-1}(f^n)) \\ (P^*, \zeta^{-1}(g^n))
    \end{pmatrix}.
\end{equation}

The matrix $\mathbf{M}_{\phi,\psi}$ has the following structure:
\begin{equation}
    \mathbf{M}_{\phi,\psi} = 
    \begin{pmatrix}
        m_{11} & m_{12} & m_{13} & m_{14} & m_{15} \\
        m_{21} & m_{22} & m_{23} & m_{24} & m_{25} \\
        m_{31} & m_{32} & m_{33} & m_{34} & m_{35} \\
        m_{41} & m_{42} & m_{43} & m_{44} & m_{45} \\
        m_{51} & m_{52} & m_{53} & m_{54} & m_{55}
    \end{pmatrix},
\end{equation}
where the entries are defined as follows:
\begin{align}
    &\text{Row 1--4, Columns 1--4:}\notag\\
    &m_{i1} = c_{gs}(F_i^*, \hat{S}^*), \quad m_{i2} = c_{ga}(F_i^*, \hat{H}^*),\notag\\
    &m_{i3} = c_{gb}(F_i^*, \hat{Q}^*), \quad m_{i4} = c_{go}(F_i^*, \hat{K}^*),\notag\\
    &\text{with } F_1^* = S^*, F_2^* = H^*, F_3^* = Q^*, F_4^* = K^*,\notag\\
    &\text{and diagonal corrections: } m_{ii} \rightarrow 1 + m_{ii} \text{ for } i = 1,2,3,4.\notag\\
   &\text{Column 5:}\notag\\
    &m_{15} = c_{go}(S^*, \hat{K}^*),\notag\\
    &m_{25} = c_{go}(H^*, \hat{K}^*),\notag\\
    &m_{35} = c_{go}(Q^*, \hat{K}^*),\notag\\
    &m_{45} = c_{go}(K^*, \hat{K}^*),\notag\\
    &m_{55} = 1-(P^*, \Lambda_{P^*}).\notag\\
    &\text{Row 5, Columns 1--4:}\notag\\
    &m_{51} = m_{52} = m_{53} = 0, \quad m_{54} = -(P^*, \Lambda_{P^*}).
\end{align}

Once these scalar values are determined, the full field variables $\phi^{n+1}$ and $\psi^{n+1}$ are recovered through substitution back into Equations~\eqref{eq:decoupled_AC}--\eqref{eq:decoupled_CH}.

\begin{remark}[CC-MSAV-BDF1 scheme differences]
We enumerate the components that differ for the CC-MSAV-BDF1 scheme. For the remaining components, the derivation follows the same format, with extrapolation terms replaced by the corresponding variables from the previous time step. 

The explicit part of the Allen--Cahn equation on the right-hand side, denoted by $f^n$, becomes:
\begin{align}
    f^n &= \frac{\phi^n}{\Delta t} - M_{\phi}\Biggl\{\left(\theta + \gamma_{\text{bend}}\frac{3\sqrt{2}}{4\varepsilon}\right)\Delta\phi^{n} \notag\\
    &\quad + \gamma_{\text{surf}} \frac{3 \sqrt{2}}{2}\left[V^n + \int_{\Omega} S\left(\phi^{n}\right)\left(-\phi^n\right) \,d\mathbf{x}\right] S(\phi^{n}) \notag\\
    &\quad + \gamma_{\text{area}} \left[U^n + \int_{\Omega} H\left(\phi^{n}\right)\left(-\phi^n\right) \,d\mathbf{x}\right] H(\phi^{n}) \notag\\
    &\quad + \gamma_{\text{bend}} \frac{3 \sqrt{2}}{8 \varepsilon} \left[W^n + \int_{\Omega} Q\left(\phi^{n}\right)\left(-\phi^n\right) \,d\mathbf{x}\right] Q(\phi^{n}) \notag\\
    &\quad + \Biggl[Z^n + \int_{\Omega} K\left(\phi^{n},\psi^{n}\right)\left(-\phi^n\right) \,d\mathbf{x} \notag\\
    &\quad\quad + \int_{\Omega} P\left(\phi^{n},\psi^{n}\right)\left(-\psi^n\right) \,d\mathbf{x}\Biggr]K(\phi^{n},\psi^{n})\Biggr\}.
\end{align}

The explicit part of the Cahn--Hilliard equation on the right-hand side, denoted by $g^n$, becomes:
\begin{align}
    g^n &= \frac{\psi^n}{\Delta t} - \lambda\Delta\psi^{n} + \nabla \cdot \left( M_{\psi}(\phi^{n}) \nabla \tilde{\nu}^{n+1} \right),
\end{align}
where $\tilde{\nu}^{n+1}$, the explicit portion of the chemical potential, is defined as
\begin{align}
    \tilde{\nu}^{n+1} &= \Biggl[Z^n + \int_{\Omega} K\left(\phi^n,\psi^n\right)\left(-\phi^n\right) \,d\mathbf{x} \notag\\
    &\quad + \int_{\Omega} P\left(\phi^{n},\psi^{n}\right)\left(-\psi^n\right) \,d\mathbf{x}\Biggr]P\left(\phi^{n},\psi^{n}\right).       
\end{align}

One of the coefficients for the operator $\chi$ is modified to:
\begin{equation}
    C_{1,\chi} = \frac{1}{\Delta t} + M_{\phi} \gamma_{\text{surf}} \frac{3\sqrt{2}\beta}{4\varepsilon}.
\end{equation}

The operator for the Cahn--Hilliard equation, $\zeta$, is modified to:
\begin{equation}
    \zeta(\psi) := \left( \frac{1}{\Delta t} - \lambda \Delta \right) \psi.
\end{equation}
\end{remark}

\begin{remark}[Differences in the Classical--MSAV--BDF2 Scheme]

We summarize the components that differ in the Classical MSAV BDF2 formulation. 
The only modifications occur in the Cahn--Hilliard part of the scheme.

First, the right-hand side $g^n$, which collects all explicit contributions, is replaced by
\begin{align}
    g^n 
    &= \frac{4\psi^{n} - \psi^{n-1}}{2\Delta t}
    + \nabla \cdot \left( M_{\psi}(\phi^{*,n+1}) \nabla \tilde{\nu}^{\,n+1} \right),
\end{align}
where $\tilde{\nu}^{\,n+1}$, the explicit portion of the chemical potential, is defined as
\begin{align}
    \tilde{\nu}^{\,n+1}
    &= -\lambda \psi^{*,n+1} \notag\\
    &\quad + \frac{1}{3} \Biggl[
        4 Z^{n} - Z^{n-1}
        + \int_{\Omega}
            K(\phi^{*,n+1},\psi^{*,n+1})
            \, (-4\phi^{n} + \phi^{n-1}) \, d\mathbf{x} \notag\\
    &\qquad\qquad
        + \int_{\Omega}
            P(\phi^{*,n+1},\psi^{*,n+1})
            \, (-4\psi^{n} + \psi^{n-1}) \, d\mathbf{x}
        \Biggr]
        P(\phi^{*,n+1},\psi^{*,n+1}).
\end{align}

Second, the linear operator associated with the Cahn--Hilliard equation is modified to
\begin{equation}
    \zeta(\psi)
    := \left(
        \frac{3}{2\Delta t}
        - \lambda \nabla \cdot \left( M_{\psi}(\phi^{*,n+1}) \nabla \right)
    \right) \psi.
\end{equation}

\end{remark}

\begin{remark}[Differences in the Classical--MSAV--BDF1 Scheme]

Similarly, we summarize the differences in the Classical MSAV BDF1 formulation relative to the CC--MSAV BDF1 scheme.

First, the right-hand side $g^n$, which collects all explicit contributions, is replaced by
\begin{align}
    g^n &= \frac{\psi^n}{\Delta t} 
    + \nabla \cdot \left( M_{\psi}(\phi^{n}) \nabla \tilde{\nu}^{\,n+1} \right),
\end{align}
where $\tilde{\nu}^{\,n+1}$, the explicit portion of the chemical potential, is defined as
\begin{align}
    \tilde{\nu}^{\,n+1}
    &= -\lambda \psi^{n} \notag\\
    &\quad + \Biggl[
        Z^n 
        + \int_{\Omega} K(\phi^{n},\psi^{n})\, (-\phi^{n}) \, d\mathbf{x} \notag\\
    &\qquad\qquad
        + \int_{\Omega} P(\phi^{n},\psi^{n})\, (-\psi^{n}) \, d\mathbf{x}
        \Biggr]
        P(\phi^{n},\psi^{n}).
\end{align}

Second, the linear operator is modified to
\begin{equation}
    \zeta(\psi)
    := \left(
        \frac{1}{\Delta t}
        - \lambda \nabla \cdot \left( M_{\psi}(\phi^{n}) \nabla \right)
    \right) \psi.
\end{equation}

\end{remark}

\begin{algorithm}[H]
\caption{BDF2 MSAV Solution Procedure for Coupled Allen--Cahn--Cahn--Hilliard System}
\begin{algorithmic}[1]
\State \textbf{Input:} Previous solutions $\phi^n, \phi^{n-1}, \psi^n, \psi^{n-1}$; SAVs $V^n, U^n, W^n, Z^n$ and their $(n-1)$ values
\State \textbf{Output:} Updated fields $\phi^{n+1}, \psi^{n+1}$ and SAVs at time $n+1$
\State
\State \textbf{Preprocessing:}
\State \quad Compute extrapolations $\phi^{*,n+1} = 2\phi^n - \phi^{n-1}$, $\psi^{*,n+1} = 2\psi^n - \psi^{n-1}$
\State \quad Evaluate basis functions $S^*, H^*, Q^*, K^*, P^*$ at extrapolated values
\State
\State \textbf{Linear System Setup:}
\State \quad Compute RHS vectors $f^n$ and $g^n$ using explicit terms
\State \quad Solve elliptic problems: $\hat{f}^n = \chi^{-1}(f^n)$ and $\hat{g}^n = \zeta^{-1}(g^n)$
\State \quad For each basis function $F \in \{S^*, H^*, Q^*, K^*, P^*\}$:
\State \quad\quad Solve $\hat{F} = \chi^{-1}(F)$ or $\Lambda_{P^*} = \zeta^{-1}(\nabla \cdot (M_{\psi}(\phi^{*,n+1}) \nabla P^*))$
\State
\State \textbf{Reduced System Solution:}
\State \quad Assemble $5 \times 5$ matrix $\mathbf{M}_{\phi,\psi}$ and vector $\mathbf{b}_{\phi,\psi}$
\State \quad Solve dense system: $\mathbf{M}_{\phi,\psi} \cdot [\tilde{B}, \tilde{C}, \tilde{D}, \tilde{E}, \tilde{G}]^T = \mathbf{b}_{\phi,\psi}$
\State
\State \textbf{Field Recovery:}
\State \quad Reconstruct $\phi^{n+1}$ using Equation~\eqref{eq:decoupled_AC} with computed integrals
\State \quad Reconstruct $\psi^{n+1}$ using Equation~\eqref{eq:decoupled_CH} with computed integrals
\State
\State \textbf{SAV Update:}
\State \quad Update SAVs using
\begin{equation}
    V^{n+1} = \tilde{B}+\frac{1}{3}\left(4V^n - V^{n-1} +  \int_{\Omega} S(\phi^{*,n+1})(- 4\phi^n + \phi^{n-1})\,d\mathbf{x}\right)
\end{equation}
 \quad and similarly for others
\State
\State \textbf{Return} $\phi^{n+1}, \psi^{n+1}, V^{n+1}, U^{n+1}, W^{n+1}, Z^{n+1}$
\end{algorithmic}
\end{algorithm}

\newpage

\section{Fast Direct Solver for the Discretized Equations}
\label{sec:appendix_fast_direct_solver}

To solve the linear system $\chi \boldsymbol{\phi} = \mathbf{f}$, we employ a fast direct solver based on spectral diagonalization.

Constant-coefficient elliptic operators on rectangular domains can be diagonalized using fast orthogonal transforms depending on the boundary conditions: FFT for periodic, DST for homogeneous Dirichlet, and DCT for homogeneous Neumann boundaries. 
These transform-based Poisson solvers are classical \cite{strang1999discrete,trefethen2000spectral,leveque2007finite,schumann1988fast}. Efficient direct solvers for separable elliptic problems have also been developed in the spectral and spectral-element settings \cite{shen1994efficient,kwan2007efficient}. 
Here, we apply the same idea to \emph{fourth-order} constant-coefficient elliptic operators discretized with cell-centered finite differences and homogeneous Neumann boundary conditions.

Since $\chi$ is a polynomial in the discrete Laplacian $\Delta_h$, it shares the same eigenvectors and is therefore diagonalizable by the DCT. 
Let $h$ denote the grid spacing and $(m,n)$ the grid size. 
The eigenvalues of $\Delta_h$ for cell-centered grids with Neumann BCs are
\begin{equation}
    \lambda_{ij} = -\frac{4}{h^2}\!\left[\sin^2\!\left(\frac{\pi i}{2m}\right) + \sin^2\!\left(\frac{\pi j}{2n}\right)\right],
    \quad i=0,\ldots,m-1,\ j=0,\ldots,n-1,
\end{equation}
with $\lambda_{00}=0$ corresponding to the constant mode.
The operator $\chi$ then has eigenvalues
\begin{equation}
    \lambda_{\chi,ij} = C_{1,\chi} - C_{2,\chi}\lambda_{ij} + C_{3,\chi}\lambda_{ij}^2,
\end{equation}
where coefficients $C_{k,\chi}$ depend on the form of $\chi$ in Eq.~\eqref{eq:chi_operator}.

The solver proceeds as:
\begin{enumerate}
    \item Apply a forward 2D DCT: $\hat{\mathbf{f}} = \text{DCT}(\mathbf{f})$.
    \item Compute $\hat{\boldsymbol{\phi}} = \hat{\mathbf{f}} / \lambda_{\chi}$, handling the zero mode by enforcing $\sum \mathbf{f} = 0$.
    \item Apply an inverse DCT: $\boldsymbol{\phi} = \text{IDCT}(\hat{\boldsymbol{\phi}})$.
\end{enumerate}

This direct $O(N^2 \log N)$ solver achieves machine precision and eliminates iterative convergence issues. 
While similar procedures apply for periodic or Dirichlet boundaries (using FFT or DST), this DCT-based formulation provides an efficient and straightforward approach to fourth-order elliptic equations with homogeneous Neumann conditions.

\section{Spatial Discretization: Complete Technical Framework}
\label{sec:appendix_spatial_operators}

In this appendix, we collect the technical details of the staggered-grid,
cell-centered finite difference discretization used in
Section~\ref{sec:spatial}.  We
record here the grid layout, discrete operators, summation-by-parts identities,
and the verification of structure-preserving properties used throughout the
paper.

\subsection{Grid Structure and Discrete Function Spaces}

We consider the rectangular domain $\Omega=(0,L_x)\times(0,L_y)$, where
$L_x=m h$ and $L_y=n h$, with uniform grid spacing $h>0$. Following standard
staggered-grid constructions, we define the one-dimensional grids
\begin{align}
E_m &= \{ i h \mid i=0,\dots,m \},\\
C_m &= \{ (i-\tfrac12)h \mid i=1,\dots,m \},\\
\bar C_m &= \{ (i-\tfrac12)h \mid i=0,\dots,m+1 \},
\end{align}
representing edge-centered points, cell centers, and cell centers including
ghost points, respectively. Analogous definitions apply in the $y$-direction.

East--west faces are indexed by $i+\tfrac12$, $i=0,\dots,m$, and north--south
faces by $j+\tfrac12$, $j=0,\dots,n$.

Based on these grids, we define the discrete function spaces
\begin{align}
\mathcal{C}_{m\times n}
&= \{ \phi : C_m\times C_n \to \mathbb{R} \},\\
\bar{\mathcal{C}}_{m\times n}
&= \{ \phi : \bar C_m\times \bar C_n \to \mathbb{R} \},\\
\mathcal{E}^{\mathrm{ew}}_{m\times n}
&= \{ f : E_m\times C_n \to \mathbb{R} \},\\
\mathcal{E}^{\mathrm{ns}}_{m\times n}
&= \{ g : C_m\times E_n \to \mathbb{R} \}.
\end{align}

In component form, we write
\[
\phi_{i,j}=\phi(x_i,y_j), \qquad
f_{i+\frac12,j}=f(x_{i+\frac12},y_j), \qquad
g_{i,j+\frac12}=g(x_i,y_{j+\frac12}),
\]
where $x_i=(i-\tfrac12)h$ and $y_j=(j-\tfrac12)h$.

\subsection{Discrete Difference and Averaging Operators}

We now introduce the discrete operators used throughout the paper.

\subsubsection{Edge-to-Center Differences}

For face-centered quantities, we define
\begin{align}
(d_x f)_{i,j}
&= \frac{f_{i+\frac12,j}-f_{i-\frac12,j}}{h},
&& f\in\mathcal{E}^{\mathrm{ew}}_{m\times n},\\
(d_y g)_{i,j}
&= \frac{g_{i,j+\frac12}-g_{i,j-\frac12}}{h},
&& g\in\mathcal{E}^{\mathrm{ns}}_{m\times n}.
\end{align}
These operators provide a conservative discretization of the divergence.

\subsubsection{Center-to-Edge Operators}

For $\phi\in\mathcal{C}_{m\times n}$, we define the averaging operators
\begin{align}
(A_x\phi)_{i+\frac12,j} &= \tfrac12(\phi_{i,j}+\phi_{i+1,j}),\\
(A_y\phi)_{i,j+\frac12} &= \tfrac12(\phi_{i,j}+\phi_{i,j+1}),
\end{align}
and the discrete gradients
\begin{align}
(D_x\phi)_{i+\frac12,j} &= \frac{\phi_{i+1,j}-\phi_{i,j}}{h},\\
(D_y\phi)_{i,j+\frac12} &= \frac{\phi_{i,j+1}-\phi_{i,j}}{h}.
\end{align}

\subsubsection{Discrete Laplacian}

We define the discrete Laplacian by composing divergence and gradient operators:
\begin{align}
(\Delta_h\phi)_{i,j}
&= (d_xD_x\phi)_{i,j}+(d_yD_y\phi)_{i,j}\\
&= \frac{1}{h^2}
\left(\phi_{i+1,j}+\phi_{i-1,j}+\phi_{i,j+1}+\phi_{i,j-1}-4\phi_{i,j}\right).
\end{align}

\subsection{Discrete Inner Products and Norms}

For cell-centered functions $\phi,\psi\in\mathcal{C}_{m\times n}$, we use the
discrete $L^2$ inner product
\[
(\phi,\psi)_h
= h^2\sum_{i=1}^m\sum_{j=1}^n \phi_{i,j}\psi_{i,j}.
\]

For face-centered quantities, we define
\begin{align}
[f,g]_{\mathrm{ew},h}
&= \tfrac{h^2}{2}\sum_{i,j}
\big(f_{i+\frac12,j}g_{i+\frac12,j}
+f_{i-\frac12,j}g_{i-\frac12,j}\big),\\
[f,g]_{\mathrm{ns},h}
&= \tfrac{h^2}{2}\sum_{i,j}
\big(f_{i,j+\frac12}g_{i,j+\frac12}
+f_{i,j-\frac12}g_{i,j-\frac12}\big).
\end{align}

We then define the discrete $H^1$ norm
\[
\|\nabla_h\phi\|_2^2
=[D_x\phi,D_x\phi]_{\mathrm{ew},h}
+[D_y\phi,D_y\phi]_{\mathrm{ns},h}.
\]

\subsection{Boundary Conditions}

Throughout the paper, we impose homogeneous Neumann boundary conditions.
We enforce these using ghost cells by setting
\[
\phi_{0,j}=\phi_{1,j}, \quad \phi_{m+1,j}=\phi_{m,j}, \quad
\phi_{i,0}=\phi_{i,1}, \quad \phi_{i,n+1}=\phi_{i,n}.
\]
With this choice, the face-centered normal derivatives vanish at the boundary:
\[
(D_x\phi)_{\frac12,j}=(D_x\phi)_{m+\frac12,j}=0, \qquad
(D_y\phi)_{i,\frac12}=(D_y\phi)_{i,n+\frac12}=0.
\]

\subsection{Summation-by-Parts Identities}

The staggered-grid operators satisfy discrete summation-by-parts identities.
For $\phi\in\bar{\mathcal{C}}_{m\times n}$,
$f\in\mathcal{E}^{\mathrm{ew}}_{m\times n}$, and
$g\in\mathcal{E}^{\mathrm{ns}}_{m\times n}$, we have
\begin{align}
[D_x\phi,f]_{\mathrm{ew},h}
&= -(\phi,d_x f)_h + B_x(\phi,f),\\
[D_y\phi,g]_{\mathrm{ns},h}
&= -(\phi,d_y g)_h + B_y(\phi,g),
\end{align}
where $B_x$ and $B_y$ denote boundary contributions.

Under the homogeneous Neumann conditions imposed above and vanishing boundary
fluxes, these boundary terms are zero. We therefore obtain the discrete Green's
identity
\[
[D_x\phi,D_x\psi]_{\mathrm{ew},h}
+[D_y\phi,D_y\psi]_{\mathrm{ns},h}
= -(\phi,\Delta_h\psi)_h
= -(\Delta_h\phi,\psi)_h.
\]
In particular,
\[
(\phi,\Delta_h\phi)_h = -\|\nabla_h\phi\|_2^2 \le 0.
\]

\subsection{Exact Mass Conservation}

Because all fluxes are discretized in conservative form, mass conservation
follows immediately.

\begin{theorem}
For any fluxes $f\in\mathcal{E}^{\mathrm{ew}}_{m\times n}$ and
$g\in\mathcal{E}^{\mathrm{ns}}_{m\times n}$ satisfying the boundary conditions,
\[
h^2\sum_{i,j}(d_x f+d_y g)_{i,j}=0.
\]
\end{theorem}

\begin{proof}
The result follows by telescoping sums in each coordinate direction. The remaining
details are straightforward and omitted.
\end{proof}

\subsection{Structure-Preserving Discretization of Nonlinear Terms}

We now verify the key structure-preserving properties used in the energy analysis.

\subsubsection{Quadratic Face Averaging}

We define the quadratic face-averaging operators
\begin{align}
A_x^{(q)}(\phi^2)_{i+\frac12,j}
&= \tfrac13(\phi_{i+1,j}^2+\phi_{i+1,j}\phi_{i,j}+\phi_{i,j}^2),\\
A_y^{(q)}(\phi^2)_{i,j+\frac12}
&= \tfrac13(\phi_{i,j+1}^2+\phi_{i,j+1}\phi_{i,j}+\phi_{i,j}^2).
\end{align}

A direct computation shows that the discrete product rule
\[
D_x(\phi^3)=3A_x^{(q)}(\phi^2)D_x\phi,
\qquad
D_y(\phi^3)=3A_y^{(q)}(\phi^2)D_y\phi
\]
holds exactly. Combining this identity with the SBP property yields
\[
(\phi^3,\Delta_h\phi)_h
= -3(\phi^2,|\nabla_h\phi|^2)_h.
\]

\subsubsection{Variable Mobility Flux}

We discretize the variable-mobility Cahn--Hilliard flux as
\[
\nabla\cdot(M_\psi\nabla\nu)
\;=\;
d_x(M_\psi(A_x\phi)D_x\nu)
+d_y(M_\psi(A_y\phi)D_y\nu).
\]
Using SBP, we obtain
\[
(\nu,\nabla\cdot(M_\psi\nabla\nu))_h
= -\|\sqrt{M_\psi}\nabla_h\nu\|_2^2 \le 0.
\]

\subsubsection{Nonlocal Bending Term}

Finally, we discretize the nonlocal bending contribution as
\begin{align}
[\phi|\nabla\phi|^2]_{i,j}
&= \phi_{i,j}\left[
\frac{|D_x\phi|^2_{i+\frac{1}{2},j} + |D_x\phi|^2_{i-\frac{1}{2},j}}{2}
+ \frac{|D_y\phi|^2_{i,j+\frac{1}{2}} + |D_y\phi|^2_{i,j-\frac{1}{2}}}{2}
\right],\\
[\nabla\cdot(\phi^2\nabla\phi)]_{i,j}
&= d_x((A_x\phi^2)D_x\phi)_{i,j}
+ d_y((A_y\phi^2)D_y\phi)_{i,j}.
\end{align}
A careful but routine application of SBP then yields the discrete bending-energy
identity stated in the main text.

\section{Selection of Stabilization Parameters}
\label{app:stabilization_parameters}

The stabilization parameters $\theta$ and $\lambda$ require careful selection to ensure energy stability, monotonic energy dissipation, and mass conservation. We conducted systematic parameter testing to determine the minimum viable values that satisfy these criteria for interface width $\varepsilon = 0.03125$ and spatial resolution $N = 256$.

\paragraph{Allen-Cahn Stabilization Parameter $\theta$}
The Allen-Cahn stabilization parameter $\theta$ provides artificial dissipation in the chemical potential to control interface dynamics. Initial validation tests over 100 time steps identified $\theta = 1.5$ as the minimum value ensuring energy monotonicity (after initial transients) and mass conservation to high precision. This value is used for all standard simulations. For robustness tests involving aggressive morphological changes (Section~\ref{sec:complex_dynamics}), we increase $\theta$ to 30 to maintain stability under more challenging interfacial dynamics.

\paragraph{Cahn-Hilliard Stabilization Parameter $\lambda$}
The Cahn-Hilliard stabilization parameter $\lambda$ appears in the IMEX splitting term $\lambda(\Delta\psi^{n+1} - \Delta\psi^{*,n+1})$, where $\psi^{*,n+1}$ is a second-order extrapolation. Initial testing identified $\lambda = 4.5 \times 10^4$ as minimally sufficient for short simulations, but extended tests with various initial geometries revealed that $\lambda$ must be larger to maintain stability across diverse scenarios. We therefore adopt $\lambda = 10^5$ for all simulations to robustly accommodate different test cases. 

\section{Generation of Smoothed Initial Conditions}
\label{sec:appendix_smoothing}

The numerical experiments in Section~\ref{sec:complex_dynamics} utilizing non-analytical initial geometries (triangle, star, incomplete hexagon, and crescent shape) require conversion from sharp interface profiles to smooth, diffuse interfaces suitable for phase field simulation. The sharp profiles, where the domain is partitioned into regions of $\phi = 1$ and $\phi = -1$, are relaxed into physically realistic and numerically stable starting conditions.

This smoothing is achieved by evolving the sharp initial condition, $\phi_{\text{sharp}}$, for a short duration using the classical Cahn-Hilliard equation with constant mobility:
\begin{equation}
    \frac{\partial \phi}{\partial t} = M \Delta \left( \frac{1}{\varepsilon}(\phi^3 - \phi) - \varepsilon \Delta \phi \right),
\end{equation}
subject to homogeneous Neumann boundary conditions. The mobility parameter $M = 0.01$ and time step $\Delta t = 1 \times 10^{-6}$ are selected to minimize interface migration while achieving numerical stability. These parameters limit the diffusive length scale to $O(\varepsilon)$ over the smoothing duration, preventing over-smoothing that would compromise the initial geometric features.

The smoothing simulation employs the same semi-implicit, first-order time-stepping scheme used in the main solver. The linear biharmonic term ($-M\varepsilon \Delta^2 \phi$) is treated implicitly, while the nonlinear term is explicit, resulting in linear, constant-coefficient biharmonic-type equations solved using the fast direct solver (~\ref{sec:appendix_fast_direct_solver}). 

The smoothing evolution is run for a short total time $T_{\text{smooth}} = 10^{-5}$. 
Over this interval, the free energy decreases rapidly at early times and then 
approaches a plateau, indicating that the diffuse interface has reached its local 
equilibrium profile while the global geometric features remain essentially 
unchanged. The resulting smoothed interfaces exhibit transition layers consistent 
with the expected diffuse-interface thickness of the Cahn--Hilliard model, without 
altering the intended geometry. As a final preprocessing step, we clamp values 
outside the physical range by applying
  $ \phi \leftarrow \max(-1, \min(1, \phi))$.

\section{Simulation Parameters}
\label{sec:appendix_simulation_parameters}
This section details the parameters used for the numerical experiments presented in Section~\ref{sec:numerical_results}. Parameters that differ between growth and shrinkage scenarios are distinguished using subscripts $g$ and $s$, respectively.

\begin{table}[H]
    \centering
    \caption{Parameter values used in benchmark simulations. Physical parameters are dimensionless.}
    \label{tab:benchmark_params}
    \begin{tabular}{|l|c|c|}
        \hline
        \textbf{Parameter} & \textbf{Symbol} & \textbf{Value} \\
        \hline \hline
        \multicolumn{3}{|c|}{\textbf{Physical Parameters}} \\
        \hline
        Surface Tension & $\gamma_{\text{surf}}$ & 1.0 \\
        Bending Rigidity (Growth) & $\gamma_{\text{bend,g}}$ & 0.05 \\
        Bending Rigidity (Shrinkage) & $\gamma_{\text{bend,s}}$ & 1 \\
        Arc Length Constraint Penalty & $\gamma_{\text{area}}$ & $5 \times 10^{4}$ \\
        Osmotic Penalty (in) & $\gamma_{\text{in}}$ & $1 \times 10^{5}$ \\
        Osmotic Penalty (out) & $\gamma_{\text{out}}$ & $1 \times 10^{5}$ \\
        Equilibrium Conc. (in, Growth) & $\psi_{\text{in},g}$ & 0.65 \\
        Equilibrium Conc. (in, Shrinkage) & $\psi_{\text{in},s}$ & 0.1 \\
        Equilibrium Conc. (out) & $\psi_{\text{out}}$ & 0.8 \\
        Osmotic Baseline Constant (in) & $\beta_{\text{in}}$ & 0 \\
        Osmotic Baseline Constant (out) & $\beta_{\text{out}}$ & 0 \\
        Interface Width & $\varepsilon$ & 0.03125 \\
        Allen-Cahn Mobility & $M_{\phi}$ & 1.0 \\
        Cahn-Hilliard Mobility Param. & $M_{0}$ & 0.5 \\
        \hline
          \multicolumn{3}{|c|}{\textbf{Common Numerical Parameters}} \\
        \hline
        Grid Resolution & $N$ & $256 \times 256$ \\
        Time Step & $\Delta t$ & $1 \times 10^{-6}$ \\
        Final Time & $T$ & $0.02$ \\       
        \hline
        \multicolumn{3}{|c|}{\textbf{Numerical Parameters (MSAV Method)}} \\
        \hline
        Allen-Cahn Stabilization  & $\theta$ & 1.5 \\
        Allen-Cahn Stabilization Parameter & $\beta$ & 0 \\
        Cahn-Hilliard Stabilization & $\lambda$ & $1 \times 10^{5}$ \\
        \hline
        \multicolumn{3}{|c|}{\textbf{Numerical Parameters (NLMG Method)}} \\
        \hline
        Convergence Tolerance & \texttt{tol} & $1 \times 10^{-8}$ \\
        Minimum Refinement Level & \texttt{minlevel} & -8 \\
        Pre-smoothing Steps & \texttt{presmooth} & 2 \\
        Post-smoothing Steps & \texttt{postsmooth} & 2 \\
        Maximum Iterations & \texttt{maxits} & 100 \\
        \hline
    \end{tabular}
\end{table}

\end{document}